\documentclass[11pt,namelimits,sumlimits,a4paper]{amsart}
\usepackage[utf8]{inputenc}
\usepackage{hyperref}
\usepackage{subfiles}
\usepackage{comment}

\usepackage{amssymb,amsmath}
\usepackage[mathscr]{eucal}
\usepackage{slashed}

\usepackage[T1]{fontenc}
\usepackage[UKenglish]{babel}
\usepackage{amsfonts}
\usepackage{fancyhdr}
\usepackage{graphicx}
\usepackage{amsmath}
\usepackage{amsthm}
\usepackage{amsmath,amscd}
\usepackage{latexsym}
\usepackage{cite}
\usepackage{amssymb,amsmath}
\usepackage{tensor}
\usepackage{multicol}
\usepackage{comment}
\usepackage[mathscr]{eucal}
\usepackage{slashed}
\usepackage{times,psfrag}
\usepackage{mathtools}
\usepackage{hyperref}
\usepackage{multirow}
\usepackage{longtable}
\usepackage{bm}
\usepackage[all]{xy}
\usepackage{fancyhdr}
\usepackage{float}
\usepackage{esint}
\usepackage{layout}
\usepackage{enumitem}
\usepackage{lmodern}
\usepackage[rgb]{xcolor}
\usepackage{tikz}
\usetikzlibrary{shapes,arrows}
\usetikzlibrary{matrix,arrows}
\usepackage{dsfont}
\usepackage{tikz-cd}
\usepackage{stmaryrd}

\renewcommand{\S}{\mathcal{S}}

\newcommand{\SL}{\operatorname{SL}}

\newcommand{\R}{\mathbb{R}}

\newcommand{\C}{\mathbb{C}}

\newcommand{\SO}{\mathrm{SO}}
\newcommand{\SU}{\mathrm{SU}}
\newcommand{\U}{\mathrm{U}}

\newcommand{\st}{\big|\;}

\newcommand{\Sp}{\mathrm{Sp}}
\renewcommand{\dim}{\mathrm{dim}}

\newcommand{\vol}{\operatorname{vol}}
\newcommand{\abs}[1]{|#1|}

\renewcommand{\phi}{\varphi}
\newtheorem{theorem}{Theorem}[section]
\newtheorem*{theorem*}{Theorem}
\newtheorem{question}{Question}
\newtheorem*{question*}{Question}
\newtheorem*{conjecture*}{Conjecture}
\newtheorem*{proposition*}{Proposition}
\newtheorem*{corollary*}{Corollary}
\newtheorem{remark}[theorem]{Remark}
\newtheorem{example}[theorem]{Example}

\newtheorem{corollary}[theorem]{Corollary}
\newtheorem{proposition}[theorem]{Proposition}

\newtheorem{lemma}[theorem]{Lemma}

\newcommand{\eval}[1]{\bigg\rvert_#1}
\begin{document}

\author{Enric Solé-Farré}
\title{The Hitchin index in cohomogeneity one nearly Kähler structures}
\address{Department of Mathematics, University College London, London WC1E 6BT, United Kingdom}
\thanks{{\tt enric.sole-farre.21@ucl.ac.uk}}
\begin{abstract}
Nearly Kähler and Einstein structures admit a variational characterisation, where the second variation is associated to a strongly elliptic operator. This allows us to associate a Morse-like index to each structure. Our study focuses on how these indices behave under the assumption that the nearly Kähler structure admits a cohomogeneity one action. Specifically, we investigate elements of the index that also exhibit cohomogeneity one symmetry, reducing the analysis to an ODE eigenvalue problem. 

We apply our discussion to the two inhomogeneous examples constructed by Foscolo and Haskins. We obtain non-trivial lower bounds on the Hitchin index and Einstein co-index for the inhomogeneous nearly Kähler structure on $S^3\times S^3$, answering a question of Karigiannis and Lotay. 
\end{abstract}
\maketitle
\section{Introduction}
Consider a complete Riemannian six-manifold $(M^6, g)$ and its associated metric cone $C(M)= (M\times \R_+, dr^2 +r^2g)$. We are interested in the case where the cone has holonomy contained in $G_2$. This condition on the cone implies the existence of a pair of stable forms $(\omega, \rho) \in \Omega^2(M)\times \Omega^3(M)$ on $M$,  with $\operatorname{Stab}(\rho)\cong \SL(3, \mathbb{C})$, that satisfy the relations
\begin{align*} g &= \omega\left( \cdot, J_\rho \cdot\right) &\frac{1}{3!} \omega^3 = \frac{1}{4} \rho \wedge  \widehat{\rho} \;,  \\
    d\omega &= 3 \rho   &d^* \rho  = 4 \omega \;,~~~
\end{align*}
where $J_\rho$ is the almost complex structure induced by the stable 3-form $\rho$. 

The first two algebraic constraints imply that $(\omega, \rho)$ define an $\SU(3)$-structure on $M$. The first of which ensures that $ \omega$ is positive and of type $(1,1)$ with respect to $J_\rho$, particularly $\omega \wedge \rho =0$. The latter two fix the torsion of the $\SU(3)$-structure (cf. \cite{GH80}) and are called the nearly Kähler equations. A manifold carrying this structure is called a nearly Kähler manifold. Since manifolds with holonomy inside $G_2$ are Ricci-flat, nearly Kähler manifolds are always Einstein with constant $\lambda=5$.

Two variational characterisations of nearly Kähler structures are known. The first is due to Hitchin \cite{Hitchin00} \cite{Hitchin01}, and the second one was introduced by the author in \cite{ESF24h}. In there, we proved that the second variation of the functional can be identified with a strongly elliptic operator using the natural indefinite perfect pairing. In particular, it has a well-defined index, which we call the Hitchin index of the nearly Kähler structure. This index is a lower bound for the Einstein co-index and plays a crucial role in the deformation theory of $G_2$-conifolds, highlighting its importance.

Explicitly, the eigenspaces contributing to the Hitchin index are identified with solutions to the eigenvalue problem
\begin{equation} \label{1}
    \mathcal{E}(\nu)=  \left\{ \beta \in \Omega^2_8 \; \Bigr| ~d^* \beta =0, \; \Delta \beta= \nu \beta \right\}
\end{equation}
for $\nu \in (0,12)$, where $\Omega^2_8 \cong \{ \beta \in \Omega^2(M) \st \omega \wedge \beta =- * \beta \}$, smooth sections of the subbundle associated to the adjoint representation of $\SU(3)$ inside $\mathfrak{so}(6)$. In the language of almost hermitian geometry, these are primitive $(1,1)$ forms.  

Only six examples of simply connected nearly Kähler manifolds are known. Four homogeneous ones: 
 \begin{multicols}{2}
\begin{itemize}
    \item $(S^6, g_{round})= G_2/ \SU(3)$, 
    \item $\C P^3 = \Sp(2)/\U(1)\times \Sp(1)$ and
    \item $S^3 \times S^3 = \SU(2)^3/\triangle \SU(2)$,
    \item $\mathbb{F}_{1,2} = \SU(3)/T^2$.
\end{itemize}
\end{multicols}
The remaining two examples were constructed in 2016 by Foscolo and Haskins \cite{FH17}. They produced new examples of nearly Kähler structures on $S^6$ and $S^3\times S^3$ using cohomogeneity one methods.

By the work of Karigiannis and Lotay \cite{LK20} (cf. \cite{MS10}), all four homogeneous nearly Kähler structures are known to have Hitchin index zero. We study the index problem in the cohomogeneity one examples produced in \cite{FH17}. To tackle this, we look for solutions to the eigenvalue problem which are also cohomogeneity one. This assumption allows us to reduce the problem to a singular ODE eigenvalue problem. We prove
\begin{theorem*}[Thm. \ref{thmindS3S3}]
The Hitchin index of the inhomogeneous nearly Kähler structure in $S^3 \times S^3$ of Foscolo and Haskins is bounded below by one. Its Einstein co-index is bounded below by four.
\end{theorem*}
The solution to the eigenvalue problem is constructed following a similar strategy to the one in \cite{FH17}.

To construct complete nearly Kähler manifolds, Foscolo and Haskins first construct two families, $\Psi_a$ and $\Psi_b$, of desingularisations of the cone singularity over $N_{1,1}$, the homogeneous Sasaki-Einstein structure on $S^2 \times S^3$. These desingularisations, called nearly Kähler halves, each has a maximal volume orbit (cf. \cite[Prop. 5.15]{FH17}). To construct a complete nearly Kähler solution, Foscolo and Haskins identify conditions under which two nearly Kähler halves can be smoothly matched along their maximal volume orbits. They then demonstrate that these matching conditions are satisfied in at least two distinct cases beyond those associated with the known homogeneous examples, yielding the new complete nearly Kähler examples.

For our eigenvalue problem \eqref{1}, we first solve the corresponding ODE on the Foscolo--Haskins halves:
\begin{theorem*}[Thm. \ref{existencesolutionhalf}]
Let $a,b>0$ and consider the nearly Kähler halves $\Psi_a$ and $\Psi_b$ of Foscolo and Haskins \cite{FH17}. Then for every eigenvalue $\Lambda\in (0,\infty)$, there exists a unique (up to scale) solution to the ODE on the nearly Kähler half $\Psi_a$ (resp. $\Psi_b$). Moreover, the solution depends continuously on $a$ (resp. $b$) and $\Lambda$.   
\end{theorem*}
After constructing these solutions, we analyse conditions on the maximal volume orbit to enable matching between halves. A key challenge here is the non-explicit nature of the Foscolo--Haskins nearly Kähler structures, which results in a non-explicit ODE. This makes obtaining the qualitative estimates needed to match the solutions quite challenging. To achieve this, we had to derive refined qualitative estimates of the Foscolo--Haskins solutions (cf. Prop. \ref{signw1w2}) and a careful analysis of the conserved quantities in the ODE.

To the best of our knowledge, there are only two similar results in the literature. The first, by Gibbons, Hartnoll, and Pope \cite{GHP03}, involves numerical approximations of the cohomogeneity one metrics constructed by Böhm \cite{Bohm98}. We believe that the methods introduced in this paper could be adapted to study the stability of cohomogeneity one Einstein metrics, thereby extending and formalising the work of \cite{GHP03}. The second result, by Branding and Siffert \cite{BS23}, examines the stability of cohomogeneity one harmonic self-maps. In their case, the underlying cohomogeneity one metrics are explicit, and the stability problem reduces to a Sturm--Liouville-type problem, neither of which occurs in our work.
\section*{Acknowledgments}
The author is indebted to his PhD supervisors, Simon Donaldson and Lorenzo Foscolo, for their continuous support, guidance and
encouragement throughout the completion of this project. We also thank Christoph Böhm for comments and suggestions on an early version of the manuscript.

This work was supported by the Engineering and Physical Sciences Research Council [EP/S021590/1]. The EPSRC Centre for Doctoral Training in Geometry and Number Theory (The London School of Geometry and Number Theory), University College London.

\section{Cohomogeneity one $\SU(3)$-structures}\label{sectionFH}
Consider $(M^6, \omega, \rho)$ a manifold carrying an $\SU(3)$-structure. By the work of Conti--Salamon \cite{CS07}, the frame bundle of any orientable hypersurface $\Sigma^5\hookrightarrow M^6$ admits a reduction to a principal $\SU(2)$-bundle. This is equivalent to the existence of a nowhere-vanishing 1-form $\eta$ and a triple of 2-forms $(\omega_1, \omega_2, \omega_3)$ satisfying 
\begin{enumerate}
    \item $\eta\wedge \omega_i \wedge \omega_j = 2 \delta_{ij} \vol_\Sigma$,
    \item $X\lrcorner \omega_1 = Y\lrcorner \omega_2 \implies \omega_3(X,Y) \geq 0$.
\end{enumerate}
The forms $\omega_i$ pointwise span a subbundle of $\Lambda^2$, which we denote as $\Lambda^2_+$. Its orthogonal complement within $\ker(\eta)$ will be denoted by $\Lambda^2_-$. This notation is justified by observing that the Hodge star restricts to $\ker(\eta)$ as $\pm \operatorname{Id}$ on $\Lambda^2_\pm$. Let $\nu$ denote the positive unit normal vector field to $\Sigma$. The induced $\SU(2)$-structure on $\Sigma$ is given explicitly by
$$ \eta= \nu \lrcorner \omega ~~~~~~~ \omega_1=\omega \mid_\Sigma ~~~~~~~ \omega_2 + i \omega_3 = \nu \lrcorner  \left( * \rho - i \rho \right) \;.$$
Conversely, given  one parameter family of $\SU(2)$-structures, we can define an $\SU(3)$-structure on $\Sigma \times (a,b)$ by taking 
\begin{equation} \label{2to3}
    \omega = \eta \wedge dt + \omega_1 ~~~~~~~~~~ \rho + i * \rho = (\omega_2 +i \omega_3) \wedge(\eta+ i dt)\;,
\end{equation}
where $t$ is a coordinate on $(a,b)$.
In our case (cf. \cite{PS10}), the principal orbit will always be diffeomorphic to $S^2\times S^3 \cong N_{1,1}= \SU(2)\times \SU(2)/ \triangle\U(1)$. 
We are interested in parameterising the set of invariant $\SU(2)$ structures on it. On $N_{1,1}$, we have a distinguished invariant $\SU(2)$ structure: the Sasaki-Einstein structure coming from the  Calabi-Yau conifold $z_1^2+z_2^2+z_3^2 + z_4^2=0$. We will denote the associated  basis of invariant forms by $\eta^{se} \in \Omega^1$, $\omega^{se}_0 \in \Omega^2_-$ and $\omega^{se}_1,\omega^{se}_2,\omega^{se}_3 \in \Omega^2_+$, satisfying 
$$d\eta^{se} = -2 \omega^{se}_1 ~~~ d\omega^{se}_2 = 3\eta^{se} \wedge \omega_3^{se} ~~~ d\omega^{se}_3 = -3\eta^{se} \wedge \omega_2^{se} ~~~ d\omega^{se}_0 = 0\;.$$
Concerning the Sasaki-Einstein structure, the space of invariant $\SU(2)$-structures on $S^2\times S^3$ is parametrised by $\R^+ \times \R^+ \times \SO_0(1,3)$. Given $(\lambda, \mu, A) \in \R^+ \times \R^+ \times \SO_0(1,3)$, the associated $\SU(2)$-structure is given by 
$$\eta =  \lambda \eta^{se} ~~~~~~~~~ \omega_i=  \mu A \omega^{se}_i\;.$$
\begin{remark}\label{remarkreeb}
    The (left-invariant) Reeb field generates the subgroup of inner automorphisms of $\SU(2)\times \SU(2)$ that fixes $\Delta \U(1)$. In terms of the invariant $\SU(2)$-structure, the Reeb field induces a rotation in the $(\omega^{se}_2,\omega^{se}_3)$-plane.
\end{remark}
We get the following formula from the structure equations for the Sasaki-Einstein metric. Let $(\lambda, \mu, A)$ denote an invariant $\SU(2)$-structure. Then,
    \begin{equation}
        d\eta= -2 \lambda \omega_1^{se}~~~~~d \omega_i = \frac{\mu}{\lambda} \eta \wedge TA \omega_i^{se} ~~~~~d (\eta \wedge \omega_i)= d\eta \wedge \omega_i = -2 \lambda \mu \langle A \omega_i, \omega_1^{se}\rangle \vol^{se}\;,
    \end{equation}
    where $T\in \operatorname{End}(\R^{1,3})$ is given by
    $T(\omega^{se}_0)=T(\omega^{se}_1)=0$, $T(\omega^{se}_2)= 3\omega^{se}_3$ and $T(\omega^{se}_3)= -3 \omega^{se}_2$.
\begin{remark}
Formula (2.17) in \cite{FH17} contains two typos, which are corrected above.
\end{remark}
\subsection{Local nearly Kähler conditions}
We can now ask under what conditions a family of $\SU(2)$ structures on $\Sigma$ gives rise to a nearly Kähler structure on $\Sigma\times (a,b)$. By the definition of a nearly Kähler structure and \eqref{2to3}, the $\SU(3)$-structure will be a nearly Kähler structure if and only if the $\SU(2)$-structure satisfies the equations
\begin{equation}\label{NHeq}
    d\omega_1= 3\eta\wedge \omega_2 ~~~~~~~~~~~ d(\eta \wedge \omega_3) =-2 \omega^2_1\;,
\end{equation}
as well as the evolution equations
\begin{equation}\label{NHevo}
    \partial_t \omega_1 = -3 \omega_3 - d\eta ~~~~~~~ \partial_t(\eta \wedge \omega_2)=- d\omega_3
    ~~~~~~~ \partial_t(\eta \wedge \omega_3) = d\omega_2 +4 \eta \wedge \omega_1\;.
\end{equation}
An $\SU(2)$-structure $(\eta, \omega_1, \omega_2, \omega_3)$ satisfying Equations \eqref{NHeq} is called a nearly hypo $\SU(2)$-structure. Equations \eqref{NHevo} are called nearly hypo evolution equations. When restricting to the case of cohomogeneity one, Foscolo and Haskins introduce a change of variables to obtain an ODE system rather than a mixed differential and algebraic system. Their results are summarised in the following proposition:

\begin{proposition}[{\cite[Prop. 3.9]{FH17}}] Let $\Psi(t)=(\lambda, \underline{u}, \underline{v})$ be a solution of the ODE system \label{241}
\begin{subequations} \label{NKsystem}
    \begin{align} 
        \lambda \partial_t u_0 + 3v_0 &=0\;, \label{5a}\\
        \lambda \partial_t u_1 + 3v_1 &= 2 \lambda^2 \;, \label{5b}\\
        \lambda \partial_t u_2 + 3v_2 &=0\;, \label{5c}\\
        \partial_t v_0 -4 \lambda u_0 &=0\;, \label{5d}\\
        \partial_t v_1 -4 \lambda u_1 &=0\;, \label{5e}\\
        \partial_t v_2 -4 \lambda u_2 &=- 3 \frac{u_2}{\lambda}\;, \label{5f}\\
        \lambda \abs{u}^2\partial_t \lambda^2- \partial_t u_2^2  &= -  \lambda^4 u_1 \label{5g}
    \end{align}
\end{subequations}
defined on an interval $(a,b)\subseteq \R$ $u_2<0$, $\lambda, \mu^2 >0$ and $u_1v_2-u_2v_1>0$. Moreover, assume that there exists some $t_0\in (a,b)$ for which the quantities
\begin{equation} \label{conservedquantities}
I_1(t)= \langle \underline{u},\underline{v} \rangle ~~~~~~~~
I_2(t)= \lambda^2 \abs{\underline{u}}^2 - u_2^2 ~~~~~~~~
I_3(t)=\lambda^2 \abs{\underline{u}}^2 - \abs{\underline{v}}^2 ~~~~~~~~
I_4(t)=v_1- \abs{\underline{u}}^2 
\end{equation}
all vanish. Then $\psi_{\lambda, \mu, A}$ with $\mu=\abs{\underline{u}}$ and 
$$ A= \frac{1}{\lambda \mu^2 }
\begin{pmatrix} 
    u_1v_2 - v_1u_2 & \lambda \mu u_0   & 0 & \mu v_0\\
    u_0v_2 - u_2v_0 & \lambda \mu u_1   & 0 & \mu v_1\\
    u_1v_0 - v_1u_0 & \lambda \mu u_2   & 0 & \mu v_2\\
    0 & 0 & -\lambda \mu^2 & 0 \\
\end{pmatrix} =
\begin{pmatrix} 
    w_0 & x_0   & 0 &  y_0\\
    w_1 & x_1   & 0 &  \frac{\mu}{\lambda}\\
    w_2 & -\lambda   & 0 &  y_2\\
    0 & 0 & -1 & 0 \\
\end{pmatrix}
$$
defines an $\SU(2) \times \SU(2)$-invariant nearly Kähler structure on $(a,b)\times N_{1,1}$.
\end{proposition}
\begin{remark}
    The vanishing of $I_1$, $I_2$, $I_3$ and $I_4$ correspond to $\omega_1\wedge \omega_3= 0$, $\omega_1^2= \omega^2_2$, and $\omega_1^2= \omega^2_3$ and the second equation of \eqref{NHeq} respectively. The ODE \eqref{5g} implies $I=(I_1,I_2,I_3,I_4)$ is a conserved quantity for the ODE.
\end{remark}
 \begin{corollary}[{\cite[Cor. 2.46]{FH17}}] \label{246}
     Let $\Psi_{\lambda, \mu, A}$ be an invariant nearly hypo structure on $N_{1,1}$ such that $w_1=0=w_2$. Then, it is an invariant hypersurface of the sine-cone over the invariant Sasaki-Einstein.
 \end{corollary}
The ODE system \eqref{NKsystem} is invariant under various symmetries. Three of them will be key in the discussion ahead: time translation $t \mapsto t+t_0$, for $t_0 \in \R$ and the following two involutions:
\begin{subequations} \label{involutions}
    \begin{align}
        \tau_1: &( \lambda,u_0, u_1, u_2, v_0, v_1, v_2, t) \mapsto (\lambda, - u_0,  - u_1, u_2, v_0, v_1, - v_2, -t)\;, \\
        \tau_2: &( \lambda,u_0, u_1, u_2, v_0, v_1, v_2, t) \mapsto (\lambda, u_0,  - u_1, u_2, -v_0, v_1, - v_2, -t)\;.
    \end{align}
\end{subequations}
A complete list of the symmetries of the ODE \eqref{NKsystem} and their geometric interpretation can be found in \cite[Prop. 3.11]{FH17}. 

There are four explicit examples of solutions to the ODE system \eqref{NKsystem}: the three homogeneous examples $S^6$, $\C P^3$ and $S^3 \times S^3$; and the sine-cone over the homogeneous $N_{1,1}$ with its homogeneous Sasaki-Einstein.
\begin{example}[The sine cone] \label{sinecone-param}
\begin{equation}
\lambda= \sin(t)~~~~~~~~
    \mu= \sin^2(t)~~~~~~~~ A= \begin{pmatrix}
1 & 0   & 0 \\
0 & \cos(t) & \sin(t) \\
0 & -\sin(t) & \cos(t)
\end{pmatrix}
\end{equation}
\end{example}
\begin{example}[Homogeneous nearly Kähler on $S^3\times S^3$]
    \begin{equation} \label{homogeneousS3S3}
    \begin{split} \lambda= 1~~~~~~~~~~~~~~~~~~~~~~~~~~~~~~~~~~~~~~~~~~~~~~~~~~
    \mu= \frac{2\sqrt{3}}{3}\sin(\sqrt{3}t) \\
    \mu A= \begin{pmatrix}
\frac{2}{3} \left(\sin^2(\sqrt{3}t) + 1\right) & \frac{\sqrt{3}}{3} \sin(2\sqrt{3}t) & \frac{2}{3} \left( 2\sin^2(\sqrt{3}t) - 1\right) \\
\frac{2}{3} \sin^2(\sqrt{3}t) & \frac{\sqrt{3}}{3}\sin(2\sqrt{3}t) & \frac{4}{3} \sin^2(\sqrt{3}t) \\
-\frac{2}{3}\cos(\sqrt{3}t) & -\frac{2\sqrt{3}}{3} \sin(\sqrt{3}t) & \frac{2}{3} \cos(\sqrt{3}t)
\end{pmatrix}
\end{split}
\end{equation} 
\end{example}
We collect some formulae and relations for general cohomogeneity one nearly Kähler manifolds:
\begin{lemma} \label{cohomdiff}
Let $(M^6, \omega, \rho)$ be a cohomogeneity one nearly Kähler manifold, and let $(\eta,\omega_i)= \psi_{\lambda, \mu, A}\left( \eta^{se}, \omega^{se}_i\right)$ the associated $\SU(2)$-moving frame. Then, we have the following relations
\begin{multicols}{2}
    \begin{enumerate} [label=(\roman*)]
    \item $d\eta = 2\frac{\lambda w_1}{\mu} \omega_0 - 2\frac{\lambda x_1}{\mu} \omega_1 -2 \omega_3$,
    \item $d \omega_0= - 3 \frac{w_2}{\lambda} \eta \wedge \omega_2$,
    \item $d(\eta\wedge \omega_0) = - 2 \frac{\lambda w_1}{\mu} \omega_1^2$,
    \item $d \omega_2 = -\frac{3 w_2}{\lambda} \eta \wedge \omega_0 - 3 \eta \wedge \omega_1  + \frac{3 y_2}{\lambda} \eta \wedge \omega_3$,
    \item $d \omega_3 = -3 \frac{y_2}{\lambda} \eta \wedge \omega_2$,
    \item $\partial_t \eta= \partial_t\log(\lambda) \eta $,
    \item $\partial_t \omega_0 = \partial_t\log(\mu) \omega_0 -2 \frac{\lambda w_1}{\mu} \omega_1 - 3 \frac{w_2}{\lambda} \omega_3$,
    \item $\partial_t \lambda = 3 y_2 -2 \frac{\lambda^2}{\mu}x_1$,
    \item $\partial_t \mu = 2 \lambda x_1$.
    \item[\vspace{\fill}]
    \end{enumerate}
\end{multicols}
\end{lemma}
\begin{proof}
All the identities follow from the above formulae and $A^{-1}=J A^t J$, since $A\in \SO(1,3)$. We only include the details of $(vii)$. We have
$$
        \partial_t \omega_0 = (\mu A )' \omega_0^{se} = \log(\mu)' \omega_0 + (A^t J A') \omega_0 = \log(\mu)' \omega_0 + \langle w', w \rangle  \omega_0+ \langle w', x \rangle  \omega_1 + \langle w', y \rangle  \omega_3 \;,
    $$
    where $\langle \cdot, \cdot \rangle$ is the standard inner product in $\R^{1,2}$.  We can now compute these inner products by knowing $\langle w, w \rangle=-1$, $\langle w, x \rangle = \langle w, y\rangle =0$ and the nearly Kähler conditions. First,
    $\langle w, w \rangle =-1$ implies $\langle w', w \rangle=0$. Now, for the $\omega_1$ component, we have
    $$
        \langle w', x \rangle = - \langle w, x' \rangle= - \frac{1}{\mu}\langle w, (\mu x)' \rangle
        =  - \frac{1}{\mu} \langle w, u ' \rangle
        = -\frac{1}{\mu} \left( -3 \langle w, v \rangle +2 \lambda w_1\right) = -2 \frac{\lambda}{\mu} w_1\;,
    $$
    where we used the structure ODE \eqref{NKsystem} and the fact that $\langle w, v \rangle=0=\langle w, u \rangle $. Similarly, 
    $$
        \langle w', y \rangle = - \langle w, y' \rangle= - \frac{1}{\lambda \mu}\langle w, (\lambda \mu y)' \rangle
        = - \frac{1}{\lambda \mu}\langle w, v' \rangle
        = -\frac{1}{\lambda \mu} \left( 4 \lambda \langle w, v \rangle - 3 \frac{u_2}{\lambda} w_2\right) = - 3 \frac{w_2}{\lambda} \;.
    $$
    where we used \eqref{NKsystem}, the fact that $\langle w, u \rangle=0$ and $u_2=-\mu\lambda$ on the last step.
\end{proof}
\noindent Similar relations could be obtained for the remaining forms, but we omit them since we will not need them in our discussion.
\begin{lemma} \label{cohomdiff2}
Let $(M^6, \omega, \rho)$ be a cohomogeneity one nearly Kähler structure, and let $(\eta,\omega_i)$ the associated $\SU(2)$-moving frame, with $\omega_i=\mu A \omega_i^{se}$ and $A\in \SO_0(1,3)$ as above. Then, the first column of $A$ satisfies the evolution equations
\begin{subequations}
    \begin{align}
        \partial_t w_0 &= -2\frac{\lambda x_0}{\mu} w_1 - 3 \frac{y_0}{\lambda} w_2\\
        \partial_t w_1 &= -2\frac{\lambda x_1}{\mu} w_1 - 3 \frac{\mu}{\lambda^2} w_2 \label{10b}\\
        \partial_t w_2 &= 2\frac{\lambda^2}{\mu} w_1 - 3 \frac{y_2}{\lambda} w_2 \label{10c}
    \end{align}
\end{subequations}
\end{lemma}
\begin{proof}
From the computations of Lemma \ref{cohomdiff}, we know that 
$$\langle w' ,w \rangle =0 ~~~~~~~~~~ \langle w' ,x \rangle =-2 \frac{\lambda}{\mu}w_1 ~~~~~~~~~~ \langle w' ,y \rangle =-3 \frac{w_2}{\lambda} \;.$$
More compactly, we write this as $A^{-1} w' = \begin{pmatrix} 0 \\ -2 \frac{\lambda}{\mu}w_1 \\ -3 \frac{w_2}{\lambda}       \end{pmatrix}\;.$
The result follows by matrix multiplication.
\end{proof}
\subsection{Smooth extensions over the singular orbit}\label{Smoothext-doub-matching}
To construct complete nearly Kähler manifolds, Foscolo and Haskins first construct two families of desingularisations of the cone singularity over $N_{1,1}$. We will refer to these as nearly Kähler halves and denote them by $\Psi_a(t)$ and $\Psi_b(t)$. In both cases, the parameter measures the size of the singular orbit. 

Let us revise how the desingularising families are constructed. Due to the work of Eschenburg and Wang \cite{EW00}, we have a good understanding of the necessary and sufficient conditions for a cohomogeneity one tensor to extend smoothly over a singular orbit. In our case, this reduces to the following lemma: 
\begin{lemma}[{\cite[Lemma 4.1]{FH17} and \cite[Prop. 6.1]{PS10}}] \label{smoothextconditions}
Let $\omega= F(t) \eta^{se} \wedge dt+ G_0(t) \omega^{se}_0+ G_1(t) \omega^{se}_1+ G_2(t)\omega^{se}_2+ G_3(t)\omega^{se}_3$ an invariant two form on $(0,T) \times N_{1,1}$. Then
\begin{enumerate}[label=(\roman*)]
    \item $\omega$ extends over a singular orbit $\SU(2)^2/ \SU(2)\times \U(1)$ at $t=0$ if and only if
        \begin{enumerate}
            \item $G_0, G_1, G_2, G_3$ are even and $F$ is odd;
            \item $G_2(0)=G_3(0)=0$ and $G_0(t)- G_1(t) = -\partial_t F(0)t^2 + O(t^4)$.
        \end{enumerate}
    \item $\omega$ extends over a singular orbit $\SU(2)^2/ \triangle \SU(2)$ at $t=0$ if and only if
    \begin{enumerate}
            \item $G_0, G_1, G_2$ are odd and $G_3, F$ are even;
            \item $G_0(t)+ G_2(t) =O(t^3)$, $G_3(t)= O(t^2)$ and $G_1(t)= 2F(0)t + O(t^3)$.
        \end{enumerate}
\end{enumerate}
\end{lemma}
Under the conditions of the lemma, the ODE system \eqref{NKsystem} gives rise to a singular ODE initial value problem. Foscolo and Haskins argue the existence and uniqueness of the solution to the ODE by formally solving it in terms of a power series and then applying a contraction mapping fixed point argument.

\begin{theorem}[{\cite[Thm. 4.4 \& 4.5]{FH17}}] $~$
 For each $a>0$, there exists a unique solution to \eqref{NKsystem} that extends smoothly over the singular orbit $\SU(2)^2/ \SU(2)\times \U(1)$, denoted by $\Psi_a(t)$. Similarly, for each $b>0$, there exists a unique solution to \eqref{NKsystem} that extends smoothly over the singular orbit $\SU(2)^2/ \triangle \SU(2)$, denoted by $\Psi_b(t)$.
\end{theorem}
The first terms of each of the Taylor expansions were worked out by Foscolo and Haskins and are collected in  Appendix \ref{appendix}.
\subsection{Complete nearly Kähler solutions} \label{cohomsolNK}
Once we have nearly Kähler halves, we need to match two such halves to construct a complete solution. In that direction, we have 
\begin{proposition}[{\cite[Prop. 5.15]{FH17}}] Let $\Psi(t)$ be a solution to \eqref{NKsystem} extends smoothly over the singular orbit. Then, $\Psi(t)$ has a unique maximal volume orbit: A unique $T_*$ exists for which the nearly hypo structure on $N_{1,1}$ has mean curvature zero.    
\end{proposition}
Thus, it is reasonable to match two nearly Kähler halves along their maximum volume orbits. Let $\Psi^1$ and $\Psi^2= \widetilde{\Psi}$ be two solutions for the system \eqref{NKsystem} with maximum volume orbit at time $T_*^1$ and $T_*^2$, and assume that the two maximal volume orbits coincide, so $\left(\lambda(T^1_*), \mu(T^1_*)\right)=\left(\tilde{\lambda}(T^2_*), \tilde{\mu}(T^2_*)\right)$. In particular, the two solutions must coincide on the maximum volume orbit up to the action of the involutions \eqref{involutions}. Acting by a time translation $\tau= T_1 + T_2 -t$ and $\tau_1$ or $\tau_2$, we consider
$$ \Psi_2^\pm (t) =  \left(\tilde{\lambda}(\tau), \mp \tilde{u}_0(\tau), - \tilde{u}_1(\tau), \tilde{u}_2(\tau),  \pm \tilde{v}_0(\tau), \tilde{v}_1(\tau), -\tilde{v}_2(\tau)\right) \;.$$
We define the two solutions
\begin{subequations}
    \begin{align}
        \Psi(t) &=  \begin{cases} \Psi_1(t) &~~ 0 \leq t\leq T_1\\ \Psi_2^+(t) &~~ T_1 \leq t\leq T_1 + T_2        \end{cases} \label{Psi_+} \;,\\ 
        \Psi(t) &=  \begin{cases} \Psi_1(t) &~~ 0 \leq t\leq T_1\\ \Psi_2^-(t) &~~ T_1 \leq t\leq T_1 + T_2        \end{cases} \label{Psi_-} \;. 
    \end{align}
\end{subequations}
If either solution is smooth, we will have a complete nearly Kähler manifold. The following lemma details the conditions for this to happen when both halves are isometric:
\begin{lemma}[{Doubling lemma, {\cite[Lemmas 5.19 \& 8.4]{FH17}}}] \label{doublingNK}
Let $a\in (0, \infty)$  and consider $\Psi_a(t)$ the corresponding nearly Kähler half with singular orbit $S^2$. Denote by $T_a$ the time of maximum volume orbit.
\begin{enumerate}
    \item If $w_1(T_a)=0$, then \eqref{Psi_-} with $\Psi_1=\Psi_2= \Psi_a$ defines a smooth nearly Kähler structure on $\C P^3$.
    \item If $w_2(T_a)=0$, then \eqref{Psi_+} with $\Psi_1=\Psi_2= \Psi_a$ defines a smooth nearly Kähler structure on $S^2 \times S^4 $.
\end{enumerate}
Similarly, let $b\in (0, \infty)$  and consider $\Psi_b(t)$ the corresponding nearly Kähler half with singular orbit $S^3$. Denote by $T_b$ the time of maximum volume orbit. If $w_1(T_b)=0$ (resp. $w_2(T_b)=0$), then  \eqref{Psi_-} (resp. \eqref{Psi_+}) with $\Psi_1=\Psi_2= \Psi_b$ defines a smooth cohomogeneity one nearly Kähler structure on $S^3\times S^3$.
\end{lemma}
We now outline the proof of Foscolo and Haskins on the structure of an inhomogeneous nearly Kähler structure on $S^3\times S^3$ using the result in Lemma \ref{doublingNK}. Consider the curve
\begin{align*}
    \beta:  (0, \infty) &\rightarrow \R^2 \\
     b &\mapsto \left( w^b_1(T_b) ,  w^b_2(T_b)\right) \;,
\end{align*}
For small $b$, the nearly Kähler half converges to the sine-cone, and so $\lim_{b \rightarrow 0} \beta = (0,0)$. Moreover, the homogeneous nearly Kähler structure on $S^3\times S^3$ corresponds to $b=1$ and $\beta(1)=\left(\frac{\sqrt{3}}{3} ,0 \right)$. Foscolo and Haskins prove
\begin{theorem}[{\cite[Thm. 7.12]{FH17}}] \label{completeS3S3}
    There exists $b_* \in (0,1)$ such that $\beta(b_*)=\left(0, w_2(b_*)\right)$. By Lemma \ref{doublingNK}, the nearly Kähler solution \eqref{Psi_-} with $\Psi_1=\Psi_2= \Psi_{b_*}$ defines a smooth nearly Kähler structure on $S^3\times S^3$.
\end{theorem}
Their proof strategy first involves relating the zeros of $w_1$ and $w_2$ with those of $v_0$ and $u_0$, respectively. The functions $u_0$ and $v_0$ satisfy the system \eqref{5a}-\eqref{5d}
$$ \lambda \partial_t u_0 = -3v_0  ~~~~~~~~~~~~~~~ \partial_t v_0 = 4 \lambda u_0\;.$$
In particular, they are amenable to a Sturm comparison argument with the Legendre Sturm-Liouville problem
$$ \sin(t) \partial_t \hat{u} = -3\hat{v} ~~~~~~~~~~~~~~~ \partial_t \hat{v} = 4 \sin(t) \hat{u}\;,$$
which is the linearisation of the system \eqref{NKsystem} on the sine-cone. In their proof, Foscolo and Haskins can only prove that the curve $\beta$ must cross the vertical axis in the range $b \in (0,1)$ but can not establish whether such crossing is unique, although they numerically conjecture this to be the case. 

In any case, there exists $b_* \in (0,1)$ for which the curve $\beta$ crosses the vertical axis for the last time before arriving at the homogeneous structure. For the remainder of the notes, we will refer to the corresponding $\Psi_{b_*}$ (and its complete double) as \textbf{the} inhomogeneous nearly Kähler structure on $S^3 \times S^3$.

We conclude this section by characterising this inhomogeneous nearly Kähler structure, which will be useful when studying its index. Although we do not have an explicit expression for $w_1(t)$ or $w_2(t)$, we can characterise their qualitative behaviour.
\begin{proposition}\label{signw1w2}
    Let $\Psi_{b_*}(t)$ be the nearly Kähler half corresponding to the inhomogeneous nearly Kähler structure on $S^3\times S^3$ described in \cite{FH17}. Then 
 \begin{enumerate}[label=(\roman*)]
    \item $w_1(t)>0$ for $t\in (0, T_*)$, and 
    \item $w_2(T_*)>0$.
 \end{enumerate}
\end{proposition}
\begin{proof}
The solution $\Psi_{b_*}$ corresponds to the last time the family $\beta(b)$ crosses the axis $w_1=0$ before the homogeneous solution \eqref{homogeneousS3S3}. Since the homogeneous solution satisfies $w_1(t)> 0$ for $t\in (0,\frac{\pi\sqrt{3}}{6}]$, it follows that $w^{b_*}_1(t)>0$ for $t\in (0,T_*)$. 

Now, this implies that $w_1(T_*)$ is a zero with a non-positive slope. Thus, Eq. \eqref{10b} reduces to 
$$\partial_t w_1 \eval{{T_*}} = -3 \frac{\mu}{\lambda^2} w_2 \leq 0 \;,$$
which implies $w_2(T_*)\geq 0$. If it were zero, we would have $w_1(T_*)=w_2(T_*)=0$, and we would be on the sine cone by Lemma \ref{246}, so $w_2(T_*)>0$, as needed.
 \end{proof}
\section{The Hitchin functional in the cohomogeneity one setting}
Recall the Hitchin functional introduced in \cite{ESF24h}:
\begin{align*}
    \mathcal{Q}: \mathcal{U} & \rightarrow \R \\
     \omega & \mapsto \frac{1}{3}\int_M \vol_{d\omega} - 4 \int_M\vol_\omega \;, 
\end{align*}
with $\mathcal{U}= \{ \omega\in \Omega^2 \st d\omega \mathrm{~stable,~}\omega \mathrm{~stable~and~ positive,~} \omega^2 \mathrm{~exact}\}$. It is instructive to investigate how the set $\mathcal{U}$ and the Hitchin functional $\mathcal{Q}$ restrict to the cohomogeneity one case. Consider $\omega= \lambda \eta^{se}\wedge dt + \underline{u} \omega^{se}$ a stable cohomogeneity one 2-form. The stability of $\omega$ corresponds to $\lambda \abs{u}^2 \neq 0 $, and one obtains similar open conditions for the stability of $d\omega$ and the positivity of $\omega$ with respect to the induced almost complex structure. Finally, the condition for $\omega^2$ to be exact can be shown to be equivalent to $d\omega^2 = 0$ in this instance.  The closedness condition corresponds to the evolution equation 
\begin{equation}
        \partial_t \abs{\underline{u}}^2 =  4 \lambda u_1\;, \label{edomega20}
\end{equation} 
 since $\omega^2 = 2 \abs{\underline{u}}^2 \vol^{se}_h + 2 \lambda \underline{u} \eta^{se} \wedge dt \wedge \omega^{se}$, and the claim follows by differentiation.
\begin{proposition}
Let $\omega = \lambda(t) \eta^{se}\wedge dt  +  \underline{u}(t) \omega^{se}$ be cohomogeneity one 2-form satisfying the evolution equation \eqref{edomega20}. The functional $\mathcal{Q}$ restricted to cohomogeneity one forms becomes 
$$\mathcal{Q}^{(1)}(\lambda, \underline{u})= C\int_I 4 \lambda^3 + \lambda \abs{\partial_t \underline{u}}^2 - 4 \lambda^2 \partial_t u_1 + \frac{9}{\lambda} (u_2^2 +u_3^2) - 12 \lambda \abs{\underline{u}}^2 dt\;,$$
for $C\in \R$ a constant.
\end{proposition}
\begin{proof} As above, let $\omega= \lambda \eta^{se} \wedge dt + \underline{u}\omega^{se}$. Then 
$$d\omega =  -2 \lambda dt\wedge \omega^{se}_1 + \left( \partial_t \underline{u}\right) dt \wedge \omega^{se} +3 u_2 \eta^{se}\wedge \omega^{se}_3 - 3u_3 \eta^{se}\wedge \omega^{se}_2\;.$$
By \cite[Prop. 3.20]{ESF24h}, $\omega\in \mathcal{U}$ carries a natural associated $\SU(3)$-structure, and $\widehat{d\omega} = *d\omega$ will be given by 
$$\widehat{d\omega} = 2 \lambda^2 \eta^{se} \wedge \omega^{se}_1 + \left( \partial_t u_0 \right) \lambda \eta^{se} \wedge \omega_0^{se} -  \sum_{i=1}^3 \left( \partial_t u_i \right) \lambda \eta^{se} \wedge \omega_i^{se} +3 \frac{u_2}{\lambda} dt \wedge \omega^{se}_3 - 3\frac{u_3}{\lambda} dt \wedge \omega^{se}_2\;.$$
Thus, we have 
$$
    \vol_{d\omega}= 2\left( 4\lambda^3 + \lambda \abs{\partial_t \underline{u}}^2  - 4\lambda^2 \partial u_1  + \frac{9}{\lambda}\left( u_2^2+u_3^2\right) \right) \eta^{se} \wedge dt\wedge \vol_h^{se}\;.$$
Using that $\vol_{\omega}$ is proportional to $\lambda\abs{\underline{u}}^2$, the claim follows from the definition of $\mathcal{Q}$ and integration along the Sasaki-Einstein fibres.
\end{proof}
Now, it is convenient to introduce a change of basis. Recall that the Reeb field induces a rotation in the span $\langle \omega_2^{se}, \omega_3^{se}\rangle$ (cf. Remark \ref{remarkreeb}). Thus, we find it suitable to introduce the new basis $(\underline{\hat{u}}, \theta)= (\hat{u}_0, \hat{u}_1, \hat{u}_2, \theta)$, related to $\underline{u}$ by
\begin{equation}\label{changeofvariables}
   \hat{u}_0= u_0 ~~~~~~ \hat{u}_1= u_1 ~~~~~~ \hat{u}_2= u_2 \cos(\theta)  ~~~~~~ \hat{u}_3= \hat{u}_2 \sin(\theta)\;.
\end{equation}
Since there is no risk of confusion, we abuse notation and set $\underline{u}=\underline{\hat{u}}$ from now on. Under this change of variables and rescaling, the functional $\mathcal{Q}^{(1)}$ becomes
\begin{equation}\label{funnctionalQ1}
    \mathcal{Q}^{(1)}(\lambda, \underline{u} , \theta)= \int_I 4 \lambda^3 + \lambda \abs{\partial_t \underline{u}}^2  + \lambda (u_2\partial_t \theta)^2- 4 \lambda^2 \partial_t u_1 + \frac{9}{\lambda} u^2_2 - 12 \lambda \abs{u}^2 dt\;,
\end{equation}
with $\underline{u}= (u_0,u_1,u_2) \in \mathcal{C}^\infty(I, \R^{1,2})$.
\begin{proposition}
    The Euler-Lagrange equations for $\mathcal{Q}^{(1)}$ are 
    \begin{subequations} \label{eulerlagrangecohomone}
        \begin{align}
            \frac{\delta \mathcal{Q}^{(1)}}{\delta \lambda} \implies ~~& 12\lambda^2 + \abs{\partial_t \underline{u}}^2-8\lambda \partial_t u<_1 - \frac{9}{\lambda^2} u^2_2 -12 \abs{\underline{u}}^2=0 \label{ellambda}\\
            \frac{\delta \mathcal{Q}^{(1)}}{\delta u_0} \implies ~~ &\partial_t(\lambda \partial_t u_0) +12 \lambda u_0 = 0 \label{elu0}\\
            \frac{\delta \mathcal{Q}^{(1)}}{\delta u_1} \implies ~~&\partial_t(\lambda \partial_t u_1) +12 \lambda u_1 - 4 \lambda \partial_t \lambda = 0 \label{elu1}\\
            \frac{\delta \mathcal{Q}^{(1)}}{\delta u_2} \implies ~~ &\partial_t(\lambda \partial_t u_2) +12 \lambda u_2 - \frac{9}{\lambda}u_2 - \lambda (\partial_t \theta)^2 u_2= 0 \label{elu2}\\
            \frac{\delta \mathcal{Q}^{(1)}}{\delta \theta} \implies ~~ & \partial_t \left( \lambda u^2_2 \partial_t \theta \right)=0 \label{eltheta}
        \end{align}
\end{subequations}
\end{proposition}
By the Principle of Symmetric Criticality of Palais \cite{Palais79}, solutions to these Euler Lagrange equations correspond to cohomogeneity one nearly Kähler solutions on $S^2\times S^3 \times I$. In particular, together with Equation \eqref{edomega20}, they should be equivalent to the system of Foscolo and Haskins \cite{FH17} above. We show this to be the case. First, we have 
\begin{lemma}
    Equations \eqref{elu0}-\eqref{eltheta} are equivalent to the system \eqref{5a}-\eqref{5f} of Foscolo--Haskins.
\end{lemma}
\begin{proof}
First, Equation \eqref{eltheta} directly implies $\lambda u^2_2 \partial_t \theta=C$ for some $C\in \R$, but boundary conditions force $C=0$, from which it follows that $\partial_t \theta =0$. So, we can choose $\theta(t)=0$ as in \cite{FH17}. The converse is immediate. Differentiation by $t$ of $\lambda \partial_t u_i$ implies that Equations \eqref{5a}-\eqref{5f} are equivalent to Equations \eqref{elu0}-\eqref{elu2}.
\end{proof}
Thus, we are left with showing that Equation \eqref{ellambda} and the condition $\partial_t \abs{\underline{u}}^2= 4\lambda u_1$ are equivalent to \eqref{5g}. Instead, we will show that the Euler-Lagrange equations are equivalent to the conservation of the quantities $I_i(t)$ from \eqref{conservedquantities}. We have the following:
\begin{proposition}
    Under Equations \eqref{elu0}- \eqref{elu2} and $\partial_t \abs{\underline{u}}^2= 4\lambda u_1$, the conditions $I_1(t)=0$ and $I_4(t)=0$ are automatically satisfied. Moreover, Equation \eqref{ellambda} is equivalent to the conditions $I_2(t)=0$. The remaining condition $I_3(t)=0$ follows from $I_2=0=I_4$. 
\end{proposition}
Before we prove the proposition, we introduce the following auxiliary calculation.
\begin{lemma}\label{auxcomputatuinpartialu}
    Under Equations \eqref{elu0}- \eqref{elu2} and $\partial_t \abs{\underline{u}}^2= 4\lambda u_1$, we have
    $$ \abs{\partial_t \underline{u}}^2 = 2 \lambda \partial_t u_1 +12 \abs{\underline{u}}^2 - \frac{9}{\lambda^2}u^2_2\;.$$
\end{lemma}
\begin{proof}
The proof is just a smart combination of the Leibniz rule and the aforementioned equations.
\begin{align*}
    \abs{\partial_t \underline{u}}^2 &= \partial_t \langle \underline{u}, \partial_t \underline{u} \rangle - \langle \underline{u}, \partial_t \frac{1}{\lambda} \left(\lambda \partial_t \underline{u}\right) \rangle =\partial_t( 2\lambda u_1) + (2 \lambda u_1) \frac{\partial_t \lambda}{\lambda} - \frac{1}{\lambda}\langle \underline{u}, \partial_t(\lambda \partial_t \underline{u})\rangle\\
    &= 2\lambda \partial_t u_1 + 4 u_1 \partial_t \lambda - \frac{1}{\lambda}\Big[ - 12 \lambda\abs{\underline{u}}^2 + 4\lambda u_1 \partial_t \lambda + \frac{9}{\lambda} u^2_2\Big] =2 \lambda \partial_t u_1 +12 \abs{\underline{u}}^2 - \frac{9}{\lambda^2}u^2_2\;. \qedhere
\end{align*}
\end{proof}
\begin{proof}[Proof of Proposition]
We start by showing that the conditions $I_1(t)=0=I_4(t)$ are automatically satisfied. First, we have
$$ -3I_1(t) = -3 \langle \underline{u}, \underline{v}\rangle = - \lambda u_0 \partial_t u_0 + \lambda u_1 \partial_t u_1 + 2 \lambda^2 u_1 +  \lambda u_2 \partial_t u_2 = \frac{\lambda}{2} \partial_t \abs{\underline{u}}^2 + 2 \lambda^2 u_1 = 0\;,  $$
by Equation \eqref{edomega20}. Similarly, for $I_4$, we have
$$3 I_4(t)= 3(v_1- \abs{\underline{u}}^2)=  2\lambda^2 - \lambda \partial_t u_1 -3\abs{u}^2\;.$$
Differentiating, and using Equations \eqref{elu1} and \eqref{edomega20}, we have 
$$ 3 \partial_t I_4  = 4 \lambda \partial_t \lambda - \partial_t(\lambda \partial_t u_1) -3 \partial_t \abs{\underline{u}}^2 = 12 \lambda f_1 - 3(4\lambda f_1)=0 \;. $$
Thus, $I_4(t)$ is a constant, which must again be 0 by the boundary conditions. Let us now study the relation between the Euler-Lagrange equation for $\lambda$ and $I_2$. Using the Lemma \ref{auxcomputatuinpartialu}, we have
\begin{align*}
    \frac{\delta \mathcal{Q}^{(1)}}{\delta \lambda} &= 12\lambda^2 + \abs{\partial_t \underline{u}}^2-8\lambda \partial_t u_1 - \frac{9}{\lambda^2} u^2_2 -12 \abs{u}^2= 12\lambda^2 - 6\lambda \partial_t u_1 - \frac{18}{\lambda^2} u^2_2 \\& = 6 \Big[ 2 \lambda^2 - \lambda \partial_t u_1 - \frac{3}{\lambda^2} u_2^2 \Big]= 6I_4(t) +18\lambda^2 I_2(t)\;.
\end{align*} 
Similarly,  for $I_3$ we have
$$I_3(t)= \lambda^2\abs{\underline{u}}^2 - \abs{\underline{v}}^2 = \lambda^2\abs{\underline{u}}^2 - \frac{\lambda^2}{9}\Big[ \abs{\partial_t \underline{u}}^2 -4 \lambda \partial_t u_1 +4 \lambda^2 \Big]\;,$$
and so, by a direct substitution in Equation \eqref{ellambda}, we get
\begin{align*}
    \frac{\delta \mathcal{Q}^{(1)}}{\delta \lambda} +\frac{9}{\lambda^2}I_3(t)&=\left(8\lambda^2-4\lambda \partial_t u_1 -12\abs{\underline{u}}^2\right)+\left(9\abs{\underline{u}}^2- \frac{9}{\lambda^2}u^2_2\right)\\ &=8I_4(t) + \frac{9}{\lambda^2}I_2(t)\;.
\end{align*}
Therefore, $I_3(t)$ will vanish if and only if $I_2$ (and $I_4$) vanishes, as needed.
\end{proof}
\section{The Hitchin index in the cohomogeneity one setting}
We now focus on the eigenvalue problem \eqref{1} associated with the Hitchin index. Given a nearly Kähler manifold $(M^6,\omega,\rho)$, we are interested in 2-forms $\beta \in \Omega^2_8$ satisfying 
\begin{equation}\label{PDE1}
    \Delta \beta = \nu \beta ~~~~~~~~~~~~~~~~~ d^*\beta=0\;,
\end{equation}
for $0< \nu <12$, with the limiting case $\nu=12$ corresponding to infinitesimal deformations of the nearly Kähler structure. There is a one-to-one correspondence with solutions to the 1st-order PDE system
\begin{equation}\label{PDE2}
    d\beta= \frac{\Lambda}{4}\gamma ~~~~~~~~~~~~~~~~~ d^* \gamma= \frac{\Lambda}{3}\beta \;,
\end{equation}
with $\gamma\in \Omega^3_{12}= \{\gamma \in \Omega^3 \st \omega \wedge \gamma=0= \rho  \wedge \gamma \}$ and $\Lambda= \pm \sqrt{12\nu}$.  In what follows, we will restrict ourselves to the positive branch of the square root and assume that $\Lambda>0$. Indeed, notice that if $(\beta, \gamma)$ is a solution to \eqref{PDE2} with $\Lambda$, then $(\beta, -\gamma)$ is a solution for $-\Lambda$, giving rise to the same solution of \eqref{PDE1}.

We will focus on the first order PDE system \eqref{PDE2} for $\Lambda\in (0,12)$. While finding the complete set of solutions to this system seems currently out of hand, even in the cohomogeneity one case, we can restrict ourselves to finding solutions to the PDE system with the same cohomogeneity one symmetry as the underlying nearly Kähler structure. In other words, we are computing the Hitchin index of the functional $\mathcal{Q}^{(1)}$ introduced above. 

For the remainder of the section, $(M^6, \omega, \rho)$ will denote a cohomogeneity one nearly Kähler. Let us start by characterising cohomogeneity one forms of type $8$ and $12$.
\begin{lemma}\label{cohomtype}
    Let $(M, \omega, \rho)$ be a $\SU(2)^2$-cohomogeneity one nearly Kähler structure and let $\eta(t), \omega_i(t)$ be the associated moving frame for the  underlying $SU(2)$ structure on $N_{1,1}$. Cohomogeneity one forms of type $\beta\in \Omega^2_8$ are parameterised by
    $$\beta = h_0 \omega_0 + h_1(2 \eta \wedge dt -\omega_1)\;.$$
    Similarly, cohomogeneity one forms of type $\gamma\in \Omega^3_{12}$  are parameterised by
    $$\gamma=f_0\eta\wedge \omega_0 + g_0 dt\wedge \omega_0 + f_2(\eta\wedge \omega_2 + dt\wedge \omega_3) + f_3(\eta\wedge \omega_3 - dt\wedge \omega_2)\;.$$
\end{lemma}
\begin{proof}
Let $\beta= h_i(t) \omega_i + V(t)\eta\wedge dt  $ be an arbitrary cohomogeneity one 2-form. The condition that $\beta$ is of type 8 is equivalent to $\omega\wedge \beta = -*\beta$. By direct computation, we have
\begin{subequations}
    \begin{align*}
        *\beta  &= - h_0 \eta \wedge dt \wedge \omega_0 + \frac{V}{2} \omega^2_1+ \sum_{i=1}^3 h_i \eta \wedge dt \wedge \omega_i\\
        \omega \wedge \beta &= h_0  \eta\wedge dt \wedge \omega_0 + \left( h_1 + V\right) \eta\wedge dt\wedge \omega_1 + h_2 \eta\wedge dt \wedge \omega_2 + h_3 \eta\wedge dt \wedge \omega_3 + h_1 \omega^2_1\;, 
    \end{align*}
\end{subequations}
which implies $h_2=h_3=0$ and $V=-2h_1$, as needed. Similarly, let  $\gamma= f_i(t) \eta\wedge \omega_i + g_i dt \wedge \omega_i$ be an arbitrary cohomogeneity one 3-form. The condition that $\gamma$ is of type 12 is equivalent to $\gamma\wedge \omega=0$ and $\gamma \wedge \rho= 0 =\gamma \wedge \hat{\rho}$. Again, by direct computation, we have
$$ \omega \wedge \gamma = f_1 \eta \wedge \omega^2_1 + g_1 dt \wedge \omega^2_1 ~~~~~~~~~~~~ \rho \wedge \gamma =*(f_3+g_2) ~~~~~~~~~~~~  \hat{\rho} \wedge \gamma = *(g_3-f_2)\;, $$
so, $f_1=g_1=0$, $f_2=g_3$ and $f_3=-g_2$, as needed.
\end{proof}
Let us now compute $d\beta$ and $d^*\gamma = -*d* \gamma$.  Using Lemma \ref{cohomdiff}, the defining equations \eqref{NHeq} and the evolution equations \eqref{NHevo}, we can compute the exterior derivative of $\beta$:
\begin{align*}
    d\beta= &~ \partial_t h_0 dt\wedge \omega_0 + h_0(d\omega_0 + dt \wedge \partial_t \omega_0) -\partial_t h_1 dt\wedge \omega_1+ h_1(2 d\eta \wedge dt -d\omega_1 - dt \wedge \partial_t \omega_1)\\
    =&\left( \frac{\partial_t (\mu h_0)}{\mu} + 6 \frac{\lambda w_1}{\mu}h_1\right)dt\wedge \omega_0 - \left( \partial_t h_1 +2h_0 \frac{\lambda w_1}{\mu} + 6 h_1 \frac{\lambda x_1}{\mu}\right)dt\wedge \omega_1 \\&- \left(\frac{3 w_2}{\lambda} h_0 +3 h_1 \right)\left(dt\wedge \omega_3 +\eta \wedge \omega_2\right)\;.
\end{align*}
Similarly, for $\gamma\in\Omega^3_{12}$ of cohomogeneity one, one computes 
$$*\gamma= -f_0 dt\wedge \omega_0 + g_0\eta \wedge \omega_0 + f_2(dt \wedge \omega_2 - \eta \wedge \omega_3) + f_3(dt \wedge \omega_3 + \eta \wedge \omega_2) \;.$$
Again, using Lemma \ref{cohomdiff} and equations \eqref{NHeq} and \eqref{NHevo}, we compute the exterior derivative
\begin{align*}
    d*\gamma =& f_0 dt\wedge d\omega_0 + \partial_t g_0 dt\wedge \eta \wedge \omega_0+ g_0\left(dt \wedge \partial_t(\eta \wedge \omega_0) + d(\eta \wedge \omega_0)\right)\\
    &+\partial_t f_2 \eta\wedge dt\wedge \omega_3 - f_2\left(dt\wedge d\omega_2+ d(\eta \wedge \omega_3) + dt\wedge \partial_t(\eta \wedge \omega_3)\right)\\
    &-\partial_t f_3 \eta\wedge dt\wedge \omega_2 - f_3\left(d(\eta \wedge \omega_2) + dt\wedge \partial_t(\eta \wedge \omega_2)-dt\wedge d\omega_3\right)\\
    =& - \left(\partial_t g_0 + g_0 \partial_t \log(\lambda \mu) +6f_2 \frac{w_2}{\lambda}+3f_3\frac{w_2}{\lambda} \right)\eta\wedge dt \wedge \omega_0\\
    &+\left( 2g_0 \frac{\lambda w_1}{\mu} -2f_2\right)\eta\wedge dt \wedge \omega_1 +\left( 3f_0\frac{w_2}{\lambda} - \partial_t f_3 -6f_3 \frac{y_2}{\lambda}\right)\eta\wedge dt \wedge \omega_2\\
    &+\left( 3g_0 \frac{w_2}{\lambda} + \partial_t f_2 + 6f_2 \frac{y_2}{\lambda}- 3f_3 \frac{y_2}{\lambda}\right)\eta\wedge dt \wedge \omega_3 +\left( 2f_2 -2g_0 \frac{\lambda w_1}{\mu}\right) \omega^2_1 
\end{align*}
Since $\Lambda\neq 0$, we immediately get $f_0=0=f_3$. Thus, the PDE system \eqref{PDE2} reduces to
\begin{subequations} \label{12}
    \begin{align}
        \partial_t (\mu h_0) + 6 h_1 \lambda w_1 &= \frac{\Lambda}{4} \mu g_0 \label{12a}\;,\\
        \partial_t h_1 +2h_0 \frac{\lambda w_1}{\mu} + 6 \frac{\lambda x_1}{\mu}  h_1&=0 \label{12b}\;,\\
        \partial_t (\lambda \mu g_0) +6 \mu w_2 f_2&= -\frac{\Lambda}{3}\lambda \mu h_0 \label{12c}\;,\\
        3\frac{w_2}{\lambda}g_0 +  \partial_t f_2 + 6\frac{y_2}{\lambda} f_2&=0 \label{12d}\;,\\
        3 \frac{w_2}{\lambda} h_0 +3 h_1 + \frac{\Lambda}{4} f_2&=0 \label{12e}\;,\\
        2\frac{\lambda w_1}{\mu}g_0 -2 f_2- \frac{\Lambda}{3}h_1&= 0\;. \label{12f}
    \end{align}
\end{subequations}
We distinguish two separate systems. Equations \eqref{12a}-\eqref{12d} form a first order ODE system for $h_0, h_1, g_0, f_2$ whilst Equations \eqref{12e}-\eqref{12f} are of order zero and linear in $h_1$ and $f_2$. In particular, for $\Lambda \neq \sqrt{72}$, we can rewrite them as
 \begin{subequations} \label{conquant12}
    \begin{align}
        h_1(t) &= \frac{1}{\Lambda^2-72}\left( 72 \frac{w_2}{\lambda} h_0 + 6\Lambda  \frac{\lambda w_1}{\mu} g_0 \right)\;,\\
        f_2(t)&=\frac{-1}{\Lambda^2-72}\left( 12\Lambda \frac{w_2}{\lambda} h_0 + 72  \frac{\lambda w_1}{\mu} g_0 \right)\;. \label{conquant2}
    \end{align}
\end{subequations}
These last two equations are conserved quantities of the ODE system:
\begin{proposition} \label{compatibilityODE-0}
Let $(h_0,h_1, g_0, f_2)$ be a solution to \eqref{12a}-\eqref{12d} such that they satisfy \eqref{12e}-\eqref{12f} for some time $t_0$. Then \eqref{12e}-\eqref{12f} are satisfied for all time that $(h_0,h_1,g_0, f_2)$ is defined.
\end{proposition}
First, we state the following technical computation:
\begin{lemma}\label{auxcomputation}
We have
\begin{subequations}
    \begin{align}
        \partial_t \left( \frac{w_2}{\lambda} h_0 \right)= \frac{\Lambda}{4}\frac{w_2}{\lambda}g_0 - 6\frac{w_1 w_2}{\mu} h_1 +2\frac{\lambda w_1}{\mu}h_0 - 6 \frac{y_2 w_2}{\lambda ^2}h_0\;,\\
        \partial_t \left( \frac{\lambda w_1}{\mu} g_0 \right)=-\frac{\Lambda}{3}\frac{\lambda w_1}{\mu} h_0 - 6 \frac{w_1w_2}{\mu}f_2 - 6\frac{\lambda^2 x_1w_1}{\mu^2}g_0 - 3 \frac{w_2}{\lambda}g_0\;.
    \end{align}
\end{subequations}
\end{lemma}
\begin{proof}
By direct calculation,
\begin{align*}
    \partial_t \left( \frac{w_2}{\lambda} h_0 \right)&= \frac{w_2}{\lambda \mu} \partial_t(\mu h_0) + h_0 \left( \frac{1}{\lambda}\partial_t w_2 -\frac{w_2}{\lambda^2}\partial_t \lambda -  \frac{w_2}{\lambda \mu} \partial_t \mu \right)\\
    &=\frac{\Lambda}{4}\frac{w_2}{\lambda}g_0 - 6\frac{w_1 w_2}{\mu} h_1 +h_0 \left[\left( 2 \frac{\lambda}{\mu}w_1 - 3\frac{y_2w_2}{\lambda^2}\right)- \frac{w_2}{\lambda^2}\left( 3y_2 -2 \frac{\lambda^2}{\mu}x_1\right)- 2\frac{w_2}{\mu}x_1\right]\\
    &=\frac{\Lambda}{4}\frac{w_2}{\lambda}g_0 - 6\frac{w_1 w_2}{\mu} h_1 +2\frac{\lambda w_1}{\mu}h_0 - 6 \frac{y_2 w_2}{\lambda ^2}h_0\;,    
\end{align*}
where we used Lemmas \ref{cohomdiff} and \ref{cohomdiff2} in the second line. Similarly,
\begin{align*}
    \partial_t \left( \frac{\lambda w_1}{\mu} g_0 \right)&= \frac{w_1}{\mu^2} \partial_t(\lambda \mu g_0) + g_0 \left( \frac{\lambda}{\mu}\partial_t w_1 -2\frac{w_1 \lambda}{\mu^2}\partial_t \mu \right)\\
    &=-\frac{\Lambda}{3}\frac{\lambda w_1}{\mu}h_0 - 6\frac{w_1 w_2}{\mu} f_2 +g_0 \left[\left( -2\frac{\lambda^2 x_1}{\mu^2}w_1 - 3\frac{w_2}{\lambda}\right)- 4\frac{\lambda^2 x_1}{\mu^2}w_1\right]\\
    &=-\frac{\Lambda}{3}\frac{\lambda w_1}{\mu}h_0 - 6\frac{w_1 w_2}{\mu} f_2 - 6\frac{\lambda^2 x_1w_1}{\mu^2}g_0 -3 \frac{w_2}{\lambda}g_0\;. ~ \qedhere    
\end{align*}
\end{proof}
\begin{proof}[Proof of Proposition]
We start with \eqref{12e}. Using the previous lemma and Equations \eqref{12b} and \eqref{12d}, we have
\begin{align*}
    \frac{1}{3}\partial_t \eqref{12e}&= \partial_t\left( \frac{w_2}{\lambda}h_0\right) + \partial_t h_1 + \frac{\Lambda}{12}\partial_t f_2 = -\frac{6}{\mu} (w_1w_2 +\lambda x_1) h_1 - 6\frac{y_2w_2}{\lambda^2}h_0 - \frac{\Lambda}{2}\frac{y_2}{\lambda}f_2\\
    &= -\frac{6}{\mu} \left(w_1w_2 +\lambda x_1 - \frac{\mu y_2}{\lambda}\right) h_1 - 2 \frac{y_2}{\lambda}\eqref{12e} =  - 2  \frac{y_2}{\lambda} \eqref{12e} \;,
\end{align*}
where in the last line, we used that $w_1w_2 +\lambda x_1 - \frac{\mu y_2}{\lambda}=0$ is the inner product of the second and third rows of the matrix $A \in \SO(1,3)$. Similarly, for \eqref{12f}, we have
\begin{align*}
    \frac{1}{2}\partial_t \eqref{12f}&= \partial_t\left( \frac{\lambda w_1}{\mu}g_0\right) - \partial_t f_2 - \frac{\Lambda}{6}\partial_t h_1 = \frac{6}{\mu} (\frac{\mu y_2}{\lambda} - w_1w_2) f_2 - 6\frac{\lambda^2 x_1 w_1}{\lambda^2}h_0 + \Lambda \frac{\lambda x_1}{\mu}h_1\\
    &= \frac{6}{\mu} \left(\frac{\mu y_2}{\lambda} - w_1w_2 -\lambda x_1\right) f_2 -3 \frac{\lambda x_1}{\mu} \eqref{12f} =  - 3 \frac{\lambda x_1}{\mu}\eqref{12f} \;. \qedhere
\end{align*}
\end{proof}
Thus, we can reduce ourselves to study the ODE system for $H=(\mu h_0, \lambda \mu g_0, h_1, f_2)$:
\begin{equation} \label{fundamentalODE}
    \partial_t H = \begin{pmatrix}
        0 & \frac{\Lambda}{4 \lambda} & - 6\lambda w_1 &0\\
        - \frac{\Lambda \lambda}{3} & 0&0 & - 6 \mu w_2 \\
        - 2 \frac{\lambda w_1}{\mu^2} &0 & -6 \frac{\lambda x_1}{\mu}& 0\\
        0& -3 \frac{w_2}{\lambda^2 \mu}&0 & -6 \frac{y_2}{\lambda}
    \end{pmatrix} H
\end{equation}
with suitable initial conditions $H(t_0)$ satisfying \eqref{12e}-\eqref{12f}. To lighten the notation and given the shape of the ODE system \eqref{fundamentalODE}, we make the following change of variables for the remainder of the discussion:
$$ \xi =  \mu h_0 ~~~~~~~ \chi =  \lambda \mu g_0 \;.$$
Our problem is quite similar to the local nearly Kähler problem \eqref{NKsystem}, where now the equations \eqref{12e}-\eqref{12f} play the role of the conserved quantities $I_i(t)=0$. Thus, it raises the question of whether there is a geometric interpretation of these quantities, as in the nearly Kähler case.

To solve \eqref{fundamentalODE}, we follow the same strategy for nearly Kähler structures: We solve \eqref{fundamentalODE} on a nearly Kähler half $\Psi$ and then find suitable matching conditions along maximum volume orbits. 

First, it is instructive to study the limiting case of the sine-cone, Example \ref{sinecone-param}. In this case, the condition $w_1(t)=w_2(t)=0$ and the conserved quantities yield the reduced system 
\begin{equation} \label{Legendre}
    \partial_t \begin{pmatrix}
         \xi \\ \chi 
    \end{pmatrix} =  \begin{pmatrix}
        0 & \frac{\Lambda}{4 \sin(t)}\\ \frac{ \Lambda ~\sin(t)}{3} & 0 
    \end{pmatrix} \begin{pmatrix}
         \xi \\ \chi  
    \end{pmatrix}
\end{equation}
with $h_1=f_2=0$ whenever $\Lambda\neq \sqrt{72}$. When $\Lambda=\sqrt{72}$, the system reduces to the ODE above but now $f_2= -\sqrt{2}h_1(t)= C \sin^6(t)$.

The ODE \eqref{Legendre} is the Legendre Sturm-Liouville problem under the change of variables $u=\sin(t)$. In particular, $\Lambda=2$ and $6$ are eigenvalue solutions to the Sturm-Lioville problem, each with a 2-dimensional eigenspace given by the corresponding Legendre polynomial of the first and second kind. 

Notice that these solutions give rise to $2$-forms $\beta$ solving the PDE problem \eqref{1}, and decaying with rate $-2$, which is precisely the rate we would expect to see if we were trying to construct a solution close to the sine-cone, as it is the rate of harmonic forms on the Stenzel metric on $T^*S^3$ and the small resolution $\mathcal{O}(-1) \oplus \mathcal{O}(-1)$ (cf. \cite[Thm. 2.27]{FH17}).
\begin{remark}
    The value $\Lambda=12$ is also a solution to the Sturm-Liouville problem, corresponding to infinitesimal deformations of the nearly Kähler structure on the sine cone. Foscolo and Haskins used this for their Sturm comparison argument, discussed in Theorem \ref{completeS3S3}.    
\end{remark}
\subsection*{Existence of solutions over Nearly Kähler halves}
We want to solve the ODE system \eqref{fundamentalODE} on the nearly Kähler halves $\Psi_a$ and $\Psi_b$ discussed above. Explicitly, we want the 2-form  
$$\beta =  - h_0 \omega_0 + h_1 \left( \omega_1 -2 \eta \wedge dt \right) = -2\lambda h_1 \eta^{se}\wedge dt + \sum_{i=0}^2 \left( -w_i \xi + u_i h_1\right) \omega_i^{se} $$
to extend smoothly over the singular orbits, where we had set $\xi= \mu h_0$. The smoothness of $\gamma$ will be guaranteed by Equations \eqref{12e} and \eqref{12f}. In virtue of Lemma \ref{smoothextconditions} and the Taylor expansions in Lemmas \ref{taylorpa} and \ref{taylorpb}, we have the following:
 \begin{proposition}\label{smoothextpab}
 The 2-form $\beta$ extends over the singular orbit $\SU(2)^2/U(1)\times \SU(2)$ if and only if $\xi$ is odd and $h_1$ is even. Moreover, equations \eqref{12e} and \eqref{12f} force $f_2$ to be even and $\chi$ to be odd. They have a Taylor expansion in the form 
 $$ \xi= 6At + O(t^3)  ~~~~~~\chi= - A \Lambda t^3+  O(t^5)   ~~~~~~h_1= -\frac{2\sqrt{3}}{3a} A + O(t^2) ~~~~~~f_2= Bt^2+ O(t^4) ~~~~~~ A,B \in \R \;.$$
Similarly, the 2-form $\beta$ extends over the singular orbit $\SU(2)^2/\triangle SU(2)$ if and only if $\xi$ and $h_1$ are even. Equations 
\eqref{12e} and \eqref{12f} imply $\chi$ is odd and $f_2$ is even. Moreover, they have a Taylor expansion of the form 
$$ \xi= A\Lambda t^2 + O(t^4)  ~~~~~~\chi= 8bA t+  O(t^3)   ~~~~~~h_1= Bt^2 + O(t^4) ~~~~~~f_2= \frac{2A}{b}+O(t^2) ~~~~~~ A,B \in \R \;.$$ \end{proposition}
 \begin{proof}
 First, when the singular orbit is diffeomorphic to $S^2$, the coefficient functions $G_i$ must be even, and $F$ is odd. Since $w_i(t)$ are odd, $\xi$ is odd too. Now, $\lambda$ is odd, so $h_1$ is even, which is compatible with $u_i$ being even. 
 The conditions $G_2(t)=G_3(t)=0$ are immediate from Lemma \ref{taylorpa}. Finally, we have the condition $G_1(t) - G_0(t) = \partial_t F(0) t^2 + O(t^4)$. Let $\xi=  At+ O(t^3)$ and $h_1= B+ O(t^2)$. By the Taylor expansion in Lemma \ref{taylorpa}, this last condition is equivalent to 
 $$ \frac{3}{2} B + \frac{\sqrt{3}}{2a}A = -3B  ~~~~  \implies  ~~~~  B= - \frac{\sqrt{3}}{9a} A\;, $$
 as needed. Now, Equations \eqref{12e} and \eqref{12f} imply the parity of $f_2$ and $\chi$. Let $\chi= Ct^3+ O(t^5) $ and $f_2=D+O(t^2)$, then the first term of the Taylor expansion of \eqref{12e} and \eqref{12f} are, respectively, 
$$  -\frac{\Lambda}{4} D= 3B + \frac{\sqrt{3}}{3a}A = 0 ~~~~~~~~~  
           C= \frac{\Lambda}{6} \frac{\sqrt{3}}{9a}B= - \frac{\Lambda}{6}A\;. $$
 For the $S^3$ case, the coefficient functions $G_i$ must be odd and $F$ even. By the same argument as above, we conclude that $\xi$ and $h_1$ must be even and have Taylor expansions of the form $\xi= At^2 +O(t^4)$ and $h_1=Bt^2+O(t^4)$ for $A, B \in \R$. 

 As before, Equations \eqref{12e} and \eqref{12f} imply the parity and decay of $f_2$ and $\chi$. Let $\chi= Ct+ O(t^3) $ and $f_2=D+O(t^2)$. Then the first term of the Taylor expansion of \eqref{12e} and \eqref{12f} are, respectively, 
$$ -\frac{\Lambda}{4} D= 3 \left( - \frac{1}{6b} \right) A ~~~ \implies ~~~ D =  \frac{2}{\Lambda b} A~~~~~~~~C= 4b^2 D =  \frac{8b}{\Lambda} A\;. \qedhere $$
 \end{proof}
 We solve the ODE \eqref{fundamentalODE} such that these conditions are satisfied. To do so, we use the following technical result.
 \begin{proposition}\label{singODE}
 Consider the singular initial value problem
 \begin{equation}
     \partial_t y(t) =  \frac{1}{t} A_{-1} y(t) +  A(t)y(t) ~~~~~~~ y(0)= y_0\;,
 \end{equation}
 where $y$ takes values in $\R^k$, $A_{-1}$ is a $k\times k$ matrix, and $A(t)$ is a $k\times k$ matrix whose entries depend smoothly on $t$ near $0$.  Then, the problem has a unique smooth solution whenever $y_0$ lies in the cone spanned by eigenvectors of $A_{-1}$ of non-negative eigenvalue. Furthermore, the solution $y(t)$ depends continuously on $A_{-1}$, $A(t)$ and $y_0$. 
 \end{proposition}
 \begin{proof}
First, note that one could appeal to the general theory of first-order singular initial value problem (c.f. \cite[Thm. 4.7.]{FH17}). However, we can solve the problem directly since our ODE is linear. 

Let $\overline{A}(t)=\frac{1}{t} A_{-1} + A(t)$, so we can rewrite \eqref{singODE} as $\partial_t y = \overline{A}y$. The solution to this ODE problem is simply $y(t) = \exp{\int^t_{t_0} \overline{A}(t)}y_0$. In the neighbourhood of $0$, $B(t) = \displaystyle \int_{t_0}^t \overline{A}(s) ds$ can be put in Jordan canonical form $B = S^{-1}JS$, so $\exp{\overline{A}(t)}= S^{-1}\exp{J}S$. For simplicity, let us assume that $\overline{A}(t)$ is already diagonalised in a neighbourhood of $0$; the general case follows. Then 
$$y(t) = \exp{\left( \int_{t_0}^t \overline{A}(t) dt \right) } y_0 = \begin{pmatrix}
    t^{\lambda_1} &0 & ...& 0\\
    0 & t^{\lambda_2}&  ...&0 \\
     & & \dots  &  \\
    0 & 0 &0& t^{\lambda_n}    
\end{pmatrix} \exp{\left(\int_{t_0}^t A(t) dt\right)} y_0\;,$$
where $\lambda_i$ are the eigenvalues of ${A_{-1}}$. Since $A(t)$ is smooth near $0$, $y(t)$ will be smooth if and only if $Y(t)= \sum_i t^{\lambda_i} y_0^i$ is; where $y_0^i$ are the coordinates of $y_0$ in the suitable eigenvector basis. But the functions $ t^{\lambda_i}$ are smooth around $0$ only if $\lambda_i \geq 0$. Therefore, picking the initial condition $y_0$ orthogonal to the negative eigenspace of $A_{-1}$ is sufficient for $y(t)$ to be smooth. Continuous dependence on the parameters follows from standard ODE theory. 
\end{proof}
We can now prove the main result of this section.
\begin{theorem} \label{existencesolutionhalf}
Let $a,b>0$ and consider the nearly Kähler halves $\Psi_a$ and $\Psi_b$  with singular orbit $S^2$ and $S^3$ respectively of Foscolo and Haskins \cite{FH17}. Then, for every $\Lambda\in (0,\infty)$, there exists a unique (up to scale) solution to the ODE system \eqref{fundamentalODE} on the nearly Kähler half $\Psi_a$ (resp. $\Psi_b$). Moreover, the solution depends continuously on $a$ (resp. $b$) and $\Lambda$.   
\end{theorem}
\begin{proof}
We consider the two cases separately.
\begin{itemize}
    \item \textbf{Desingularisation over $S^2$:} In view of Proposition \ref{smoothextpab}, it is useful to consider $\overline{H}= \left(\overline{\xi}, \overline{\chi}, \overline{h_1}, \overline{f_2}\right) = \left(t^{-1}\xi,t^{-3}\chi, h_1, t^{-2}f_2\right)$. Under this reparameterisation, the ODE system \eqref{fundamentalODE} becomes
    \begin{equation}
        \partial_t \overline{H}=
        \begin{pmatrix} -\frac{1}{t} & \frac{\Lambda}{4 \lambda} t^2& - 6\frac{\lambda w_1}{t} &0\\
        - \frac{\Lambda \lambda}{3t^2} & -\frac{3}{t} &0 & - 6\frac{\mu w_2}{t} \\
        - 2 \frac{\lambda w_1}{\mu^2} t &0 & -6 \frac{\lambda x_1}{\mu}& 0\\
        0& -3 \frac{w_2}{\lambda^2 \mu} t&0 & -6 \frac{y_2}{\lambda} -\frac{2}{t}
    \end{pmatrix} H \;.
    \end{equation}   
    Using the Taylor expansions in Lemma \ref{taylorpa}, it is straightforward to check that we are under the hypotheses of Proposition \ref{singODE}, with singular term 
    \begin{equation}
        A_{-1}= 
    \begin{pmatrix} -1 & 0 & -3\sqrt{3}a & 0 \\
                    - \frac{\Lambda}{2} & -3 & 0 & 0\\
                    -\frac{\sqrt{3}}{3a} &0 & -3 & 0\\
                    0 & - \frac{2\sqrt{3}}{9a} & 0 & -6                
    \end{pmatrix}  \;.
    \end{equation}
    The matrix $A_{-1}$ has three distinct negative eigenvalues: $-6,-4$ and $-3$, and a one-dimensional kernel spanned by $\left( 6, -\Lambda, -\frac{2\sqrt{3}}{3a}, \frac{\sqrt{3} \Lambda}{27a}\right)$. 
    \item \textbf{Desingularisation over $S^3$:} As before, we consider $\overline{H}= \left(\overline{\xi}, \overline{\chi}, \overline{h_1}, \overline{f_2}\right) = \left(t^{-2}\xi,t^{-1}\chi, t^{-2}h_1, f_2\right)$. Under this reparameterisation, the ODE system \eqref{fundamentalODE} becomes
    \begin{equation}
        \partial_t \overline{H}=
        \begin{pmatrix} -\frac{2}{t} & \frac{\Lambda}{4 \lambda t}& - 6\lambda w_1&0\\
        - \frac{\Lambda \lambda}{3}t & -\frac{1}{t} &0 & - 6\frac{\mu w_2}{t} \\
        - 2 \frac{\lambda w_1}{\mu^2} &0 & -6 \frac{\lambda x_1}{\mu} - \frac{2}{t}& 0\\
        0& -3 \frac{w_2}{\lambda^2 \mu} t&0 & -6 \frac{y_2}{\lambda}
    \end{pmatrix} H \;.
    \end{equation}   
    Using the Taylor expansions in Lemma \ref{taylorpa}, the singular term is
    \begin{equation}
        A_{-1}= 
    \begin{pmatrix} -2 & \frac{\Lambda}{4b} & 0& 0 \\
                    0 & -1 & 0 & 4b^2\\
                    -\frac{1}{2b} &0 & -5 & 0\\
                    0 & \frac{1}{2b^2} & 0 & -2                
    \end{pmatrix}  \;.
    \end{equation}
    The matrix $A_{-1}$ has three distinct negative eigenvalues: $-5,-3$ and $-2$, and a one dimensional kernel spanned by $\left(\Lambda, 8b, -\frac{ \Lambda}{10 b}, \frac{2}{b}\right)$. \end{itemize}
    By Proposition \ref{singODE}, the statement follows.
\end{proof}
\subsection*{Doubling and matching}
For convenience, we fix the scale parameter to $C=1$ for the remainder of the discussion. We now derive conditions under which our solutions over each nearly Kähler half can be matched along the maximum volume orbit to produce an element for the cohomogeneity one example.  

A pair $(\beta,\gamma)\in \Omega^2_8 \times \Omega^3_{12}$ solving the PDE \eqref{PDE2} is given by 
\begin{align*}
\beta & = -2 \lambda h_1 \eta^{se}\wedge dt + \sum_{i=0}^2  \left( w_i \xi  + u_i h_1 \right)\omega^{se}_i\\
\gamma & = \sum_{i=0}^2 \left( w_i \chi  + v_i f_2 \right) \eta^{se} \wedge \omega^{se}_i - \mu  f_2  dt \wedge\omega^{se}_3 \;.
\end{align*}
Recall from the discussion in Section \ref{cohomsolNK}, the cohomogeneity one complete nearly Kähler structure is constructed by matching a solution $\Psi(t)$ with another solution $\Psi^\pm(t)$, which is defined by a time translation and the appropriate action of the involutions \eqref{involutions}. First, we have
\begin{lemma}\label{coeffinvolutions}
Under the symmetries $\tau_1$ and $\tau_2$ we have 
$$\begin{pmatrix} 
    w_0 & u_0   &  v_0\\
    w_1 & u_1   &  v_1\\
    w_2 & u_2   &  v_2\\
\end{pmatrix}   \xRightarrow{~\tau_1~}  
\begin{pmatrix} 
    w_0 & -u_0   &  v_0\\
    w_1 & -u_1   &  v_1\\
    - w_2 & u_2   & -v_2\\
\end{pmatrix} ~~~~~~~~~~
\begin{pmatrix} 
    w_0 & u_0  &  v_0\\
    w_1 & u_1  &  v_1\\
    w_2 & u_2  &  v_2\\
\end{pmatrix}   \xRightarrow{~\tau_2~}  
\begin{pmatrix} 
    w_0 & u_0   &  -v_0\\
    -w_1 & -u_1 &  v_1\\
    w_2 & u_2   & -v_2\\
\end{pmatrix}\;.$$
\end{lemma}
We want to find conditions for which the pair $(\beta, \gamma)$ can be matched along the maximum volume orbit in the doubling case. We need to understand how a solution $H(t)=(\xi, \chi, h_1, f_2)$ behaves under these symmetries.
\begin{proposition}
    Let $H(t)$ be a solution over a nearly Kähler half $\Psi(t)$ solving \eqref{fundamentalODE} for  $\Lambda\in \R$. Then 
    \begin{enumerate}[label=(\roman*)]
        \item The tuple $H^+(t)=(-\xi, \chi, h_1, f_2)$ is a solution to \eqref{12} over the nearly Kähler half $\Psi^+(t)$.
        \item The tuple $H^-(t)=(\xi, -\chi, h_1, f_2)$ is a solution to \eqref{12} over the nearly Kähler half $\Psi^-(t)$.
    \end{enumerate}
\end{proposition}
\begin{proof}Straightforward computation. 
\end{proof}
We are now ready to match two solutions in the case of doubling a nearly Kähler half. Let $\beta^+(t)$ (resp. $\gamma^+(t)$) be the image of $\beta$ (resp. $\gamma$) under the symmetry $\tau_1$. Along the maximal volume orbit, i.e., at $t=T_*$,the functions $w_2$, $u_0$, $u_1$ and $v_2$ vanish, and so we have
\begin{align*}
\beta(T_*) & = -2 \lambda h_1 \eta^{se} \wedge dt + (\xi w_0) \omega_0^{se} +(\xi w_1) \omega_1^{se} +(h_1 u_2) \omega_2^{se}\;,\\
\beta^+(T_*)&= -2 \lambda h_1 \eta^{se} \wedge dt - (\xi w_0) \omega_0^{se} -(\xi w_1) \omega_1^{se} +(h_1 u_2) \omega_2^{se}\;,\\
\gamma(T_*) & = (w_0 \chi - v_0 f_2)\eta^{se} \wedge \omega^{se}_0  +(w_1 \chi - v_1 f_2)\eta^{se}  \wedge \omega^{se}_1  -  \mu f_2 dt \wedge \omega^{se}_3  \;, \\
\gamma^+(T_*) &= (w_0 \chi - v_0 f_2) \eta^{se} \wedge \omega^{se}_0 + (w_1 \chi - v_1 f_2) \eta^{se} \wedge \omega^{se}_1  -  \mu f_2 dt \wedge \omega^{se}_3  \;.
\end{align*}
It follows that the equation $(\beta, \gamma)= \alpha(\beta^+, \gamma^+)$ has two non-trivial solutions:
$$ \alpha=1 ~ \implies ~ \xi(T_*)=0 ~~~~~~~~~~~~~~~~~~~~~~ \alpha=-1 ~ \implies ~ \chi(T_*)=h_1(T_*)=f_2(T_*)=0\;.$$
Similarly, let $\beta^-(t)$ and $\gamma^-(t)$ be the images of $\beta$  and $\gamma$ under the involution $\tau_2$. Along the maximal volume orbit, the functions $w_1$, $u_1$, $v_0$ and $v_2$ vanish, and so we have
\begin{align*}
\beta(T_*) & = -2 \lambda h_1 \eta^{se} \wedge dt + (\xi w_0 + u_0 h_1) \omega_0^{se} +(\xi w_2 + u_2 h_1) \omega_1^{se}\;,\\
\beta^-(T_*)&= -2 \lambda h_1 \eta^{se} \wedge dt + (\xi w_0 + u_0 h_1) \omega_0^{se} +(\xi w_2 + u_2 h_1) \omega_1^{se}\;,\\
\gamma(T_*) & = (w_0 \chi) \eta^{se} \wedge \omega^{se}_0  +  (v_1 f_2) \eta^{se} \wedge \omega^{se}_1 +  (w_2 \chi)\eta^{se}  \wedge \omega^{se}_2  -  \mu f_2 dt \wedge \omega^{se}_3  \;, \\
\gamma^-(T_*) &= - (w_0 \chi) \eta^{se} \wedge \omega^{se}_0  +  (v_1 f_2) \eta^{se} \wedge \omega^{se}_1 +  - (w_2 \chi)\eta^{se}  \wedge \omega^{se}_2  -  \mu f_2 dt \wedge \omega^{se}_3 \;.
\end{align*}
As before, the equation $(\beta, \gamma)= \alpha(\beta^-, \gamma^-)$ has two non-trivial solutions:
$$ \alpha=1 ~ \implies ~ \chi(T_*)=0 ~~~~~~~~~~~~~~~~~~~~~~ \alpha=-1 ~ \implies ~ \xi(T_*)=h_1(T_*)=f_2(T_*)=0\;.$$ 
Thus, we have proved
\begin{proposition}\label{doublingS3S3FH}
Let $\Psi_{b^*}$ be the nearly Kähler half corresponding to the inhomogeneous nearly Kähler structure in $S^3\times S^3$ of Foscolo and Haskins. A solution to \eqref{fundamentalODE} extends to the whole $S^3\times S^3$ if and only if $\chi(T_*)=0$ or $\xi(T_*)=h_1(T_*)=f_2(T_*)=0$.
\end{proposition}
We can simplify the matching conditions by using the conserved quantities \eqref{12e}-\eqref{12f}.
\begin{lemma} \label{rmkchi0-xi0}
Assume that the nearly Kähler half doubles under $\tau_2$, so $w_1(T_*)=0$. Then
    \begin{itemize}
        \item If $h_1 (T_*) =0$, we have $\xi(T_*)=f_2(T_*)=0$.
        \item If $\Lambda \neq 0$ and $f_2(T_*)=0$, we have $\xi(T_*)=h_1(T_*)=0$.
        \item If $\Lambda\neq \sqrt{72}$ and $\xi(T_*)=0$, we have $h_1(T_*)=f_2(T_*)=0$. 
    \end{itemize}
A similar statement holds for the involution $\tau_1$.
\end{lemma} 
\begin{proof}
    We only show the details for the case $w_1(T_*)=0$. The constraints \eqref{12e}-\eqref{12f} at $t=T_*$ reduce to 
    \begin{subequations}
      \begin{align}
        3 \frac{w_2}{\lambda \mu} \xi +3 h_1 + \frac{\Lambda}{4} f_2&=0 \;,\\
        -2 f_2- \frac{\Lambda}{3}h_1&= 0\;.
      \end{align}  
    \end{subequations}
If $h_1(T_*)=0$, the claim follows directly, since $w_2(T_*)\neq 0$, as otherwise we would be on the sine-cone. If $\Lambda \neq 0$ and $f_2(T_*)=0$, the second equation implies $h_1(T_*)=0$, and so $\xi(T_*)=0$. Finally, the determinant of the matrix 
$$\begin{pmatrix}
    3 & \frac{\Lambda}{4}\\
    \frac{\Lambda}{3} & 2
\end{pmatrix}$$
is non-zero whenever $\Lambda \neq \sqrt{72}$, and the final claim follows.
\end{proof}
Similarly, one can investigate the matching conditions when the two halves are not isometric, for instance, as is the case for the inhomogeneous $S^6$ case. However, in this case, the conditions will depend on the values of $\Psi(T_*)$ on the maximum volume orbits, which are not explicit in the case of the inhomogeneous nearly Kähler structure on $S^6$. We will not investigate this case further in these notes.

\section{Index of the inhomogeneous $S^3\times S^3$}
We study the Hitchin index of the inhomogeneous nearly Kähler structure on $S^3 \times S^3$ constructed in \cite{FH17} and discussed in Section \ref{cohomsolNK}. We will prove that in the eigenvalue range $\Lambda\in (0, \sqrt{72})$, there exists a unique complete solution to \eqref{fundamentalODE}. In particular, its nearly Kähler Hitchin index is at least one.

We prove this by performing an analysis of the zeros of the functions $\xi(t,\Lambda), \chi(t,\Lambda), h_1(t,\Lambda)$ and $f_2(t,\Lambda)$, in the same spirit to the one used in Proposition \ref{signw1w2}. First, we can exploit the dependence of the conserved quantities on $\Lambda$ when restricted to the maximum volume orbit. We have
\begin{lemma} \label{lemmasigns2}
Let $\Psi_{b_*}$ the nearly Kähler half corresponding to the inhomogeneous nearly Kähler structure, with maximum volume orbit time $T_*$. Then,
\begin{enumerate}[label=(\roman*)]
    \item the functions $h_1(T_*, \Lambda)$ and $f_2(T_*, \Lambda)$ always have opposite signs;
    \item the functions $\xi(T_*, \Lambda)$ and $f_2(T_*, \Lambda)$ have the same signs if $\Lambda< \sqrt{72}$ and have have opposite signs for $\Lambda >\sqrt{72}$.
\end{enumerate}
\end{lemma}
\begin{proof}
    Since $w_1(T_*)=0$, the conserved quantities on the maximum volume orbit evaluate to 
        \begin{subequations}
            \begin{align}
                \frac{w_2}{\lambda \mu} \xi &= -h_1 -\frac{\Lambda}{12}f_2 = \left( \frac{\Lambda^2}{72} -1 \right) h_1 \label{32a} \;,\\
                f_2 &= -\frac{\Lambda}{6} h_1\;. 
            \end{align}
        \end{subequations}
    Thus, $(i)$ follows. By Proposition \ref{signw1w2} $(ii)$, $w_2(T_*)>0$, and the second claim follows.  
\end{proof}
We can now prove a crucial result that will allow us to prove the main theorem of this section. The key idea is that $\Lambda= \sqrt{72}$ is a degenerate value for the conserved quantities that acts as a barrier.
\begin{proposition}
Let $H(t, \Lambda)=\left(\xi, \chi, h_1, f_2\right)$ be a non-trivial solution to \eqref{fundamentalODE} over $\Psi_{b_*}$.
\label{keylemma}
\begin{enumerate}[label=(\roman*)]
    \item At $\Lambda= \sqrt{72}$, we have $\xi(T_*, \sqrt{72})=0$ and $f_2(T_*, \sqrt{72})\neq 0$.
    \item The zero of $\xi(T_*, \Lambda)$ at $\Lambda= \sqrt{72}$ is transverse. In particular $\chi(T_*, \sqrt{72})\neq 0$.
    \item The function $\xi(T_*, \Lambda)$ has no zeros for $\Lambda\in (0, \sqrt{72})$.
\end{enumerate}
\end{proposition}
\begin{proof}
On the maximum volume orbit, the system \eqref{fundamentalODE} evaluates to 
\begin{subequations}
    \begin{align}
        \partial_t \xi\Bigr|_ {T_*} &= \frac{\Lambda}{4\lambda}\chi \label{SLza}\;,\\
        \partial_t \chi \Bigr|_ {T_*} &= - \frac{\Lambda \lambda}{3}\xi - 6 \mu w_2 f_2\label{SLzb}\;,\\
        \partial_t h_1 \Bigr|_ {T_*} &= 0 \label{SLzc}\;,\\
        \partial_t f_2 \Bigr|_ {T_*} &= -3 \frac{w_2}{\lambda^2 \mu} \chi \label{SLzd}\;,
    \end{align}
\end{subequations}
where all the functions on the right-hand side are evaluated at the maximum volume orbit time. 
    \begin{enumerate}[label=(\roman*)]
        \item By Equation \eqref{32a} we have $\xi(T_*, \sqrt{72})=0$. If $f_2(T_*, \sqrt{72})=0$, one of the two zeros would have to be degenerate by Lemma \ref{keylemma}. If $\xi$ (resp. $f_2$) had a non-transverse zero, Equation \eqref{SLza} (resp. Equation \eqref{SLzd}) would imply $\chi(T_*, \sqrt{72})=0$, so $H(T_*, \sqrt{72})=0$. Since $T_*$ is a smooth point of a linear first-order ODE for $H$, the uniqueness of solutions would force the solution to be trivial, leading to a contradiction.

    \item By the previous item, $\xi(T_*, \Lambda)$ changes sign at $\sqrt{{72}}$. Thus, 
    $\partial_t\xi\Bigr|_ {T_*} =0$  would force $\partial_t^2 \xi\Bigr|_ {T_*} =0$, By Equation \eqref{SLza} and \eqref{SLzb}, we have 
    $$ \partial_t^2 \xi \Bigr|_ {T_*} = \frac{\Lambda}{4\lambda} \partial_t \chi = - \frac{3\Lambda}{2\lambda} \mu w_2 f_2 \neq 0\;. $$

    \item Assume we had $\overline{\Lambda} \in (0, \sqrt{72})$ such that $ \xi(T_*,\overline{\Lambda})=0$. Then, the conserved quantities would force $f_2(T_*,\overline{\Lambda})=h_1(T_*,\overline{\Lambda})=0$. By Lemma \ref{lemmasigns2}, $\xi(T_*, \Lambda)$ and $f_2(T_*, \Lambda)$ have the same sign for $\Lambda \in (0, \sqrt{72})$, the slopes of their zeros must have the same sign. Since $w_2(T_*)>0$ by Proposition \ref{signw1w2} $(ii)$, Equations \eqref{SLza} and \eqref{SLzd} force $\chi(T_*,\overline{\Lambda})=0$. Therefore $H(T_*, \sqrt{72})=0$, which is a contradiction as above.  \qedhere
    \end{enumerate}
\end{proof}
Let us study the behaviour of these functions for small $\Lambda$.  \newpage First, we have 
\begin{lemma} \label{chi0positive}
    The function $\chi(t, \Lambda)$ is strictly positive for $t\in (0, T_*]$ and $\Lambda$ small.
\end{lemma}
\begin{proof}
    For $\Lambda=0$, the ODE system totally decouples. In particular, $\chi(t, 0)$ is a solution to the singular ODE:        
    \begin{equation}
        \partial_t \chi = -6 \frac{w_1 w_2}{\mu} \chi \label{harmonicS3S3-1}\;.
    \end{equation}
The asymptotics in Proposition \ref{smoothextpab} imply that $\chi(t, 0)> 0$ for any small time. Since Equation \eqref{harmonicS3S3-1} is linear, $\chi(t, 0)>0$ for all time. In particular, since $\chi(t,\Lambda)$ is continuous on $\Lambda$, we have $\chi(t,\Lambda)>0 $ for $\Lambda$ small.     
\end{proof}
We can use this result to refine the last statement of Proposition \ref{keylemma}:
\begin{proposition} \label{positivexi}
    The function $\xi(T_*, \Lambda)$ is strictly positive for $\Lambda \in (0, \sqrt{72})$.
\end{proposition} 
\begin{proof}
In virtue of Proposition \ref{keylemma} $(iii)$, it suffices to prove this for $\Lambda$ small enough. We argue by contradiction. Assume that $\xi(T_*,\Lambda)<0$ for some small $\Lambda$. Lemma \ref{lemmasigns2}  implies $h_1(T_*, \Lambda)>0$. By, the smoothness conditions in Theorem \ref{existencesolutionhalf}, $h_1(t, \Lambda)<0$ for $t$ small enough. In particular, there exists $T< T_*$ for which $h_1(T, \Lambda)=0$. Assume $T$ is the smallest time for which this happens. Thus, $h_1$ has a non-negative slope at $T$. By Equation \eqref{12b}, we have
$$ \partial _t h_1 \Bigr|_ {T} = - 2 \frac{ \lambda w_1}{\mu^2} \xi \geq 0\;,$$
By Proposition \ref{signw1w2} $(i)$, $w_1(t)>0$ for $t\in (0, T_*)$, so $\xi(T) \leq 0$. Again, by the smoothness conditions, $\xi(t, \Lambda)>0$ for $t$ small and so, there exists $\overline{T}\leq T$ such that $\xi(\overline{T}, \Lambda)=0$ and $\partial_t \xi \mid _{\overline{T}} \leq 0$. But Equation \eqref{12a} would imply
$$\partial_t \xi \Bigr|_ {\overline{T}} = \frac{\Lambda}{4\lambda} \chi -  6 \lambda w_1 h_1 >0\;,$$
since $\chi(t, \Lambda)>0$ by the above lemma and $h_1\leq 0$ since $\overline{T}\leq T$, which is a contradiction, and so we must have $\xi(T_*, \Lambda)>0$ for $\Lambda$ small.
\end{proof}
We now have the tools to prove the existence of a complete solution for $\Lambda \in (0, \sqrt{72})$. 
\begin{proposition} On the inhomogeneous nearly Kähler structure in $S^3 \times S^3$ of Foscolo and Haskins, there exists a cohomogeneity one solution to \eqref{PDE2} for $ \Lambda \in (0, \sqrt{72})$.
\end{proposition}
\begin{proof}
 Since $\xi(T_*, \Lambda)$ is strictly positive for $\Lambda \in (0, \sqrt{72})$, the transverse zero at $\Lambda=\sqrt{72}$ from Proposition \ref{keylemma} $(i)$ must have strictly negative slope, so $\chi(T_*, \sqrt{72})<0$ by Equation \eqref{SLza}. By Lemma \ref{chi0positive}, $\chi(T_*, 0) > 0$, and by continuity on $\Lambda$, there exists $\Lambda_*\in (0, \sqrt{72})$ such that $\chi(T_*, \Lambda_*)=0$.

The doubling conditions in Proposition \ref{doublingS3S3FH} imply that $H(t,\Lambda_*)$ doubles to a solution of the system \eqref{12} on the whole manifold.
 \end{proof}
Our main theorem is now straightforward.
\begin{theorem}\label{thmindS3S3}
    The Hitchin index of the inhomogeneous nearly Kähler structure on $S^3 \times S^3$ is bounded below by 1. The Einstein co-index is bounded below by 4.
\end{theorem}
\begin{proof}
 Due to the relation between the PDEs \eqref{PDE1} and \eqref{PDE2}, the proposition above implies that there exists a cohomogeneity one 2-form $\beta\in \Omega^2_{8,coclosed}$ solving $\Delta \beta = \nu \beta$ for $\nu \in (0, 6)$. The claim for the bound on the Hitchin index follows.
 
 The Einstein co-index bound follows from the formula of  Schwahn \cite[Lemma 3.2]{Sch22}:
\begin{equation} \label{einstein_index_NK}
     \operatorname{Ind}^{EH} = b^2(M) + b^3(M) + 3 \sum_{\nu \in (0,2)} \dim \; \mathcal{E}(\nu) + 2 \sum_{\nu \in (2, 6)}\dim \; \mathcal{E}(\nu) +  \sum_{\nu \in (6,12)} \dim \; \mathcal{E}(\nu)\;,
 \end{equation}
where $\mathcal{E}(\nu)$ are the corresponding eigenspaces; $\mathcal{E}(\nu) = \left \{ \beta \in \Omega^2_{8} ~ \st  d^* \beta =0 \;, ~ \Delta \beta = \nu \beta \right \}$.
\end{proof}
\section{Outlook}
Let us briefly discuss how one might improve the index estimate above. There are essentially two scenarios that need to be considered:\medskip

\begin{itemize}
    \item \textbf{Other cohomogeneity one solutions:} \medskip
\end{itemize}
First, in the range $\Lambda \in (0, \sqrt{72})$, we know that all cohomogeneity one solutions to \eqref{PDE2} will satisfy $\chi(T_*)=0$, in virtue of Proposition \ref{keylemma} $(iii)$. Moreover, it is not hard to see that $\chi(t, \sqrt{72})$ must have exactly one transverse zero before the maximum volume orbit. Therefore, the only phenomenon that could produce additional solutions to the index problem would be that the zero oscillates around the maximum volume orbit.  
 
Secondly, in our proof of Theorem \ref{thmindS3S3}, it was key that the conserved quantities \eqref{12e}-\eqref{12f} change behaviour at $\Lambda=\sqrt{72}$, effectively allowing us to treat $\sqrt{72}$ as a barrier. The range $(\sqrt{72},\infty)$ has no further changes in behaviour, and there is no clear strategy beyond numerical approximations to study the range $\Lambda \in (\sqrt{72}, 12)$. \medskip
\begin{itemize}
\item \textbf{Non-cohomogeneity one solutions:} \medskip
\end{itemize}
There is no reason one should expect all solutions to be cohomogeneity one. In general, we only know that the isometry group $\SU(2)^2$ acts naturally on the eigenspaces $\mathcal{E}(\nu)$. By Schur's lemma, we can decompose them as a direct sum of irreducible representations $\mathcal{E}(\nu, \rho)$, where $\rho \in \mathbb{N}^2$ is the highest weight of the representation. Cohomogeneity one solutions correspond to $\mathcal{E}(\nu, 0,0)$. 

Since the Laplacian is strongly elliptic, each irreducible representation will have an associated minimum eigenvalue, which we denote by $\nu_\rho$. We ask
\begin{question}
    Can we establish a lower bound for $\nu_\rho$?
\end{question}
One expects $\nu_\rho$ to grow with the norm of $\rho$. We outline a general strategy using the approach of Moroianu and Semmelmann in the homogeneous case \cite{MS10}. By the Weitzenböck formula, we obtain 
$$\Delta \beta =  \nabla^* \nabla  \beta - Q(\omega, \rho) \beta  =  - \partial^2_t \beta - \sum_i \nabla_{e_i} \nabla_{e_i} \beta  - Q(\omega, \rho) \beta =  \nu \beta \;, $$
where $\{e_i\}$ are an orthonormal basis of the tangent space of the fibres $S^2 \times S^3$. 

Each $\SU(2)$-structure on the fibres will be homogeneous under the action of $\SU(2)^2$, though they will not necessarily be naturally reductive in general. Nevertheless, one expects that $\nabla ^* \nabla$ can be related to the Casimir operator of $\SU(2)^2$ and that the remaining terms can be estimated by terms growing at most linearly on $\rho$. Hence, the Casimir operator dominates for large $\rho$. 

Admittedly, this heuristic is somewhat vague. One should get further insight by studying the homogeneous examples, whose spectrums are known by Moroianu and Semmelmann \cite{MS10}, and the sine-cone case, using a separation of variables argument.

\newpage
\appendix

\section{Taylor expansions}
\label{appendix}
\begin{lemma}[\cite{FH17}]\label{taylorpa}
The first few terms of the Taylor expansion of $\Psi_a$ are 
\begin{subequations}
    \begin{align*}
        \lambda(t)=& \frac{3}{2}t - \frac{2a^2+3}{12a^2}t^3+ O(t^5)\;,\\
        \mu(t)=&\sqrt{3}at+ \frac{\sqrt{3}}{9a}(3 - 7a^2)t^3+ O(t^5)\;,\\
        u_0(t)=&a^2-3a^2t^2 +O(t^4)\;,\\
        u_1(t)=&a^2-\frac{3}{2}(2a^2-1)t^2 +O(t^4)\;,\\
        u_2(t)=&-\frac{3 \sqrt{3}}{2}a t^2 + \frac{\sqrt{3}(16a^2-3)}{12a} t^4 +O(t^6)\;,\\
        v_0(t)=&3a^2t^2-\left(\frac{1}{4} +\frac{14}{3}a^2\right)t^4 +O(t^6)\;,\\
        v_1(t)=&3a^2t^2+\left(2-\frac{14}{3}a^2\right)t^4 +O(t^6)\;,\\
        v_2(t)=&\frac{3\sqrt{3}}{2}at^2 - \frac{\sqrt{3}(34a^2-3)}{12a}t^4+O(t^6)\;,\\
        w_0(t)=&\frac{\sqrt{3}}{3} a t^{-1} - \frac{\sqrt{3}}{54 a} \left(64a^2-39\right)t + O(t^3)\;,\\
        w_1(t)=&\frac{\sqrt{3}}{3} a t^{-1} - \frac{2\sqrt{3}}{27 a} \left(16a^2-3\right)t + O(t^3)\;,\\
        w_2(t)=& \frac{1}{2} t + \frac{9-76a^2}{54a^2}t^3 +O(t^5)\;.
    \end{align*}
\end{subequations}
\end{lemma}
\begin{lemma} [\cite{FH17}]\label{taylorpb}
The first few terms of the Taylor expansion of $\Psi_b$ are 
\begin{subequations}
    \begin{align*}
        \lambda(t)=&b- \frac{9}{10} \frac{b^2-1}{b}t^2+ O(t^4)\;,\\
        \mu(t)=&2bt+\frac{1}{10b}t^3 + O(t^5) \;,\\
        u_0(t)=&2b^2t-\frac{1}{5}(17b^2+3)t^3 +O(t^5)\;,\\
        u_1(t)=&2bt -\frac{23b^2-3}{5b} t^3 +O(t^5)\;,\\
        u_2(t)=&-2b^2t+\frac{1}{5}(17b^2-12)t^3 +O(t^5)\;,\\
        v_0(t)=&-\frac{2}{3}b^3+ 4b^3t^2+O(t^4)\;,\\
        v_1(t)=&4b^2t^2+\frac{2}{5}t^4 +O(t^6)\;,\\
        v_2(t)=&\frac{2}{3}b^3- b(4b^2-3)t^2+O(t^4)\;,\\
        w_0(t)=&\frac{b}{3}t^{-1} - \frac{16b^2-29}{15b} t + O(t^3)\;,\\
        w_1(t)=& t + O(t^3)\;,\\
        w_2(t)=&-\frac{b}{3}t^{-1} + \frac{32b^2-13}{30b}t +O(t^3)\;.
    \end{align*}
\end{subequations}
\end{lemma}

\bibliographystyle{alpha}
\bibliography{Bibliography}

\end{document}